\documentclass[11pt]{amsart}
\usepackage[dvips]{graphicx}
\usepackage{color}
\pdfoutput=1

\def\NN{\mathbb N}
\def\CC{\mathbb C}
\def\Chat{\hat {\mathbb C}}
\def\HH{\mathbb H}
\def\QQ{\mathbb Q}
\def\RR{\mathbb R}
\def\ZZ{\mathbb Z}
\def\RR{\mathbb R}

\def\a{\alpha}
\def\b{\beta}
\def\d{\delta}
\def\g{\gamma}
\def\i{{\bf i}}

\def\k{\kappa}

\def\s{\sigma}
\def\t{\tau}
\def\th{\theta}

\def\cal{\mathcal}
\def\C{{\cal C}}

\def\G{{\cal G}}

\def\M{{\cal M}}

\def\S{{\cal S}}

\def\M{{\cal M}}
\def\P{{\cal P}}
\def\R{{\cal R}}

\def\T{{\cal T}}
\def\Th{{\rm Th}}

  \def\arc{\mathop{\rm{arc}}} 
 \def\Arg{\mathop{\rm{Arg}}} 
  \def\arg{\mathop{\rm{Arg}}} 
    
   \def \inf{\mathop{\rm{inf}}} 
 \def\ML{\mathop{{ ML}} }
  \def\MLQ{\mathop{{ ML_{\QQ}}} }
 \def\PML{\mathop{{PML}}}
 \def\mod{\mathop{\rm{mod}}}

 \def\tr{\mathop{\rm{Tr}}}  

 \def\TS{{\rm Teich}(\Sigma)}
\def\teich{{\cal T}}
\def\Teich{Teichm\"uller }

\def\dd{\partial}
\def\ep{\epsilon}

\def\tw{ \rm{tw}}

\newtheorem{theorem}{Theorem}[section]

\newtheorem{lemma}[theorem]{Lemma}
\newtheorem{proposition}[theorem]{Proposition}
\newtheorem{corollary}[theorem]{Corollary}
\newtheorem{remark}[theorem]{Remark}

\newtheorem{introthm}{Theorem}

\begin{document}
\title{The Maskit embedding of the twice punctured torus}

 \begin{abstract}
The Maskit embedding $\cal M$ of a surface $\Sigma$ is the space of geometrically finite groups on the boundary of quasifuchsian space for which the `top' end is homeomorphic to $\Sigma$, while the `bottom' end consists of two triply punctured spheres, the remains of $\Sigma$ when two fixed disjoint curves have been pinched. As such representations vary in the character variety, the conformal structure on the top side varies over the Teichm\"uller space $\T(\Sigma)$.

We investigate $\cal M$ when $\Sigma$ is a  twice punctured torus, using the method of pleating rays. Fix a projective measure class $[\mu]$ supported on closed curves on $\Sigma$. The pleating ray $\P_{[\mu]}$ consists of those groups in $\cal M$ for which the bending measure of the top component of the convex hull boundary of the associated $3$-manifold is in $[\mu]$.  It is known that $\cal P$ is a real 1-submanifold  of $\cal M$.  Our main result is a formula for the asymptotic direction of $\cal P$ in $\cal M$ as the bending measure tends to zero, in terms of natural parameters for  the 2-complex dimensional representation space $\cal R$ and the Dehn-Thurston coordinates of the support curves to $[\mu]$ relative to the pinched curves on the bottom side. This leads to a method of locating $\cal M$ in 
$\cal R$.
\medskip

\noindent {\bf MSC classification:}   30F40, 30F60, 57M50 
\end{abstract}

\author{Caroline Series} 

\address{\begin{flushleft} \rm {\texttt{C.M.Series@warwick.ac.uk \\http://www.maths.warwick.ac.uk/$\sim$masbb/} }\\ Mathematics Institute, 
 University of Warwick \\
Coventry CV4 7AL, UK \end{flushleft}}
 \date{\today}
\maketitle

\section{Introduction}
\label{sec:introduction}

Pictures of various slices and embeddings of one dimensional Teichm\"uller spaces  
 into $\CC$ have become familiar in recent years. 
 A common feature is the complicated fractal boundary which has been studied by various authors for example~\cite{miyachi},~\cite { Indra}.  Such examples are always based on the once punctured torus and its close relatives. This paper presents for the first time  a method which makes  viable the prospect of plotting 
  a deformation space associated to a higher genus surface. The project  immediately introduces many difficulties: such a deformation space will  intrinsically have at least $2$ complex dimensions and the underlying combinatorics of the curve complex is not that of the Farey tesselation associated once punctured torus.

  The example we choose is the Maskit embedding  of the twice punctured torus, in which the representation variety  is smooth of complex dimension $2$. The key idea is explicitly  to locate the \emph{pleating rays}, that is, the loci in the representation variety along which the projective class  of the bending measure of the convex hull boundary  is fixed.
  These lines are certain branches of the solution set of a family of equations where traces of various elements in the group take real values.  To explain in more detail, we first consider the analogous embedding for the once punctured torus $\Sigma_{1,1}$.

The Maskit embedding of $\Sigma_{1,1}$ was initially explored experimentally
 by Mumford and Wright, see~\cite{Wright, Indra}.  The detailed study~\cite{kstop} by the author and Linda Keen introduced the concept of pleating rays which justified these computational results.   As explained in more detail below, these rays were used to   plot  Figure~\ref{fig:maskit}.
 The lined region, which repeats periodically with period $2$ in both directions,  indicates all representations 
$\rho: \pi_1(\Sigma_{1,1}) \to SL (2,\CC)$ whose image $G$ is free and discrete and for which one fixed essential non-peripheral simple curve $\g_{\infty} \in \pi_1(\Sigma_{1,1})$ is accidentally parabolic. The parameter $\mu \in \CC$ is  essentially the trace of another fixed curve $\g_{0}  $ which together with $\g_{\infty}$ generates $\pi_1$. After suitable normalisation, this is enough to determine a representation $\rho$. The resulting hyperbolic $3$-manifold $\HH^3/G$ is geometrically finite and homeomorphic
to $ \Sigma_{1,1} \times \RR$.  Its end invariants $\omega^{\pm}$ are both Riemann surfaces, representing the conformal structures on the quotients of the components of the regular set by $G$. One end  $\omega^-$ is conformally a triply punctured sphere, corresponding to the surface $\Sigma_{1,1}$ with the fixed curve $\g_{\infty}$ pinched. The other end $\omega^+$
is a Riemann surface homeomorphic  to $\Sigma_{1,1}$ and can thus be viewed a point in $\teich (\Sigma_{1,1})$. By standard Ahlfors-Bers theory, each point in $\teich$ is represented up to conjugation by exactly one such group $G $. The \emph{Maskit embedding} is the map 
$ \teich \to \CC$ which takes a surface to the $\mu$-parameter of the  group $G$ which represents it.
In Figure~\ref{fig:maskit}, the parameter  $ i \mu = \tr \gamma_0$ has been chosen so that the embedding is as close to the identity  map $\teich (\Sigma_{1,1}) =  \HH^2 \hookrightarrow \CC$ as possible. The embedding repeats periodically under translation $ \mu \mapsto \mu + 2$.  

\begin{figure}[hbt] 
\centering 
\includegraphics[height=6cm]{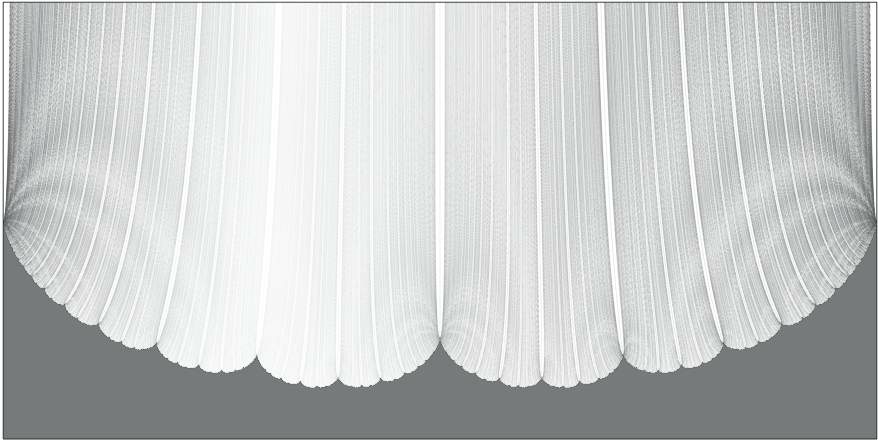} 
\caption{The Maskit embedding for the once punctured torus, showing one period in the upper half  $\mu$-plane. 
The image of $\teich (\Sigma_{1,1})$ is filled by the light gray pleating rays. Picture courtesy David Wright.}
\label{fig:maskit}
\end{figure} 

 \medskip
 
This paper lays the foundation for computing  the analogous picture of the Maskit embedding when $\Sigma = \Sigma_{1,2}$ is a twice punctured torus.    The relevant component of the representation variety $\R(\Sigma)$   is smooth of complex dimension $2$. Thus our eventual aim is to locate
the image $\M$  of the Maskit embedding of the Teichm\"uller space $  \teich(\Sigma)$ in $\CC^2$. As for $\Sigma_{1,1}$, we will do this by  locating the pleating rays, which in this case are  real $1$-submanifolds of $\M$ along which the projective class  of the bending measure of the component $\dd \C^+/G$ of the convex hull boundary facing  $\omega^+$   is supported on a fixed pair of disjoint closed curves on $\Sigma$.  In general, the pleating ray  is a connected non-singular branch  of the real algebraic variety along which the traces of the support curves take real values, see Theorem~\ref{thm:cscoords}. The main results of this paper identify the correct branch by determining its direction as the parameters of the representation  tend to infinity, equivalently as the 
bending measure tends to zero, see Theorem~\ref{thm:intromain}. The idea is that $\M$ can then be plotted by following these real trace branches until one of the supporting curves becomes parabolic, see Section~\ref{sec:examples}.

 \medskip
 
Before stating our main theorems, we briefly review our previous results from~\cite{kstop}.  As is well known, the simple closed curves on $\Sigma_{1,1}$ can be enumerated by the rationals $\QQ \cup \infty$. 
 Normalize so that the exceptional pinched curve $\g_{\infty}$ is represented by $\infty$. 
 There is one ray $\P_{p/q}$ for each $p/q \in \QQ$, representing all curves $\g_{p/q}$ whose images in $\M$ are loxodromic.
 At any point on this ray, $\dd \C^+/G$  is a pleated surface bent along $\g_{p/q}$
 while the component $\dd \C^-/G$  facing $\omega^-$ can be viewed as a copy of $\Sigma$ bent along $\g_{\infty}$ with bending angle $\pi$.  If a closed curve $\g_{p/q}$ is a bending line, then its complex length is real, so that $\g_{p/q}$ is purely hyperbolic on $\P_{p/q}$.  The trace, hence also the complex length, of $\g_{p/q}$ has no critical points on $\P_{p/q}$. It follows  that $\P_{p/q}$ is a totally real $1$-submanifold embedded in $\M$ and that the hyperbolic length $l_{p/q}$ of $\g_{p/q}$ is strictly monotonic along $\P_{p/q}$ with range $(0,\infty)$.
As $l_{p/q} \to 0$ along   $\P_{p/q}$ we approach the boundary $\dd \M$,  arriving at an algebraic limit which is the doubly cusped group in which $l_{p/q}  =0$. As $l_{p/q} \to \infty$ on the other hand,  there is no algebraic limit and the sequence of representations diverges. One of the main results of~\cite{kstop} is that $\P_{p/q}$  is asymptotic to the line $\Re \mu = 2p/q$ as $l_{p/q} \to \infty$,  identifying it uniquely among branches of $\tr \g_{p/q} \in \RR$.
 
 \medskip
 
 Turning now to the twice punctured torus $\Sigma=\Sigma_{1,2}$, suppose we have a geometrically finite free and discrete representation for which $M = \Sigma \times \RR$. Fix elements $S_1,S_2 \in \pi_1(\Sigma)$ corresponding to disjoint non-homotopic closed curves $\s_1,\s_2$ which form a maximal pants decomposition of $\Sigma$ and neither of which individually disconnects $\Sigma$.    We consider groups for which the conformal end  $\omega^-$ is a union of triply punctured spheres glued across punctures corresponding to  $\s_1,\s_2$,  while $\omega^+$ is a marked Riemann surface homeomorphic to $\Sigma$. In Section~\ref{sec:concrete} we give an explicit parameterisation of a holomorphic family of representations  $\rho: \pi_1(\Sigma) \to G(\tau_1,\tau_2), (\tau_1,\tau_2 ) \in \CC^2$, such that, for suitable values of the parameters, $G(\tau_1,\tau_2)$ has  the  above geometry.  The Maskit embedding is the map  which sends a point $X \in \teich(\Sigma_{1,2} )$ to the point  $(\tau_1,\tau_2 ) \in \CC^2$ for which the group  $ G(\tau_1,\tau_2)$ has $\omega^+ = X$.   Denote the image of this map by $\M = \M(\Sigma_{1,2})$.

  \medskip
 
Given a projective measured lamination $[\xi]$ on $\Sigma$, the \emph{pleating ray} $\P_{[\xi]}$ is the set of groups   in $\M$ for which the bending measure $\b(G)$  of $\dd \C^+/G$ is in the class $[\xi] $. 
 We  restrict to  pleating rays for which $[\xi]$ is \emph{rational},  that is, supported on closed curves, and for simplicity write $\P_{\xi}$ in place of $\P_{[\xi]}$, although noting that $\P_{\xi}$ depends only on $[\xi]$.   From general results of Bonahon and Otal~\cite{BonO} (see Theorem~\ref{thm:bo} in Section~\ref{sec:pleating}),  for any pants decomposition $\g_1,\g_2$ such that $\s_1,\s_2, \g_1,\g_2$ are mutually non-homotopic and fill up $\Sigma$, and any pair of angles $\theta_i \in (0,\pi)$, there is a unique group in $\M$  for which the bending measure of $\dd \C^+/G$ is
  $\sum_{1,2} \theta_i \delta_{\g_i}$.  (This extends to the case $\th_i = 0$  provided $\s_1,\s_2, \g_j$ fill up $\Sigma$, see Section 3,  also for the case $\th = \pi$.)
 Thus  given $\xi =\sum_{1,2}a_i \delta_{\g_i}$, there is a unique group $G = G_{\xi}(\th) \in \M$ with bending measure $\b(G)=\theta \xi$ for any sufficiently small $\theta >0$. 
 
  \medskip  
 
  In contrast to the quasifuchsian situation studied in~\cite{smallbend}, there is no algebraic limit along $\P_{\xi}$ as $\th \to 0$, see Corollary~\ref{cor:noalglim}. Intuitively this is because the groups
$ G_{\xi}(\th)$ want   to limit  on a Fuchsian group, which is however impossible since  the bending angles on the parabolic pinched curves are fixed as $\pi$. 
 Our main result is a formula for the asymptotic direction of 
 $\P_{\xi}$ in  $ \M \subset \CC^2$  in terms of the global  linear coordinates  for measured laminations on $\Sigma_{2,1}$ set up in~\cite{kps}. These  coordinates, 
 called here \emph{canonical coordinates}, assign to a measured lamination $\xi $ a point 
  $\i(\xi) = (q_1(\xi),p_1(\xi), q_2(\xi),p_2(\xi) ) \in (\RR_+ \times \RR)^2$, see Section~\ref{sec:simplecurves}.  
The coordinates  of a simple closed curve are integral; they are essentially its Dehn-Thurston coordinates relative to  $\s_1,\s_2$ and are  a close analogue of the $p,q$ coordinates for curves on $ \Sigma_{1,1}$ above. In particular, $q_i(\g) = i(\g,\s_i)$ where $i(\cdot , \cdot)$ is the usual geometric intersection number. If $\xi   =  \sum_{1,2}a_i \delta_{\g_i}$, the above Bonahon-Otal condition on 
$\s_1,\s_2, \g_1,\g_2$ is equivalent to  
$q_i(\xi) >0, i=1,2$.  We call such laminations \emph{admissible}, see Section~\ref{sec:pleating}. We call a pair of curves  $\gamma_1,\g_2$ \emph{exceptional} if 
 $q_1(\g_1)q_2(\g_2) = q_1(\g_2)q_2(\g_1)$, and we say $\xi   =  \sum_{1,2}a_i \delta_{\g_i}$ is exceptional if 
 $a_i >0, i=1,2$ and  the pair $\g_1,\g_2$ is exceptional.
The main result of this paper is:

  \begin{introthm}
   \label{thm:intromain}
Suppose that $\xi   =  \sum_{1,2}a_i \delta_{\g_i}$ is admissible and not exceptional.
Then as the bending measure $\beta(G)  \in \RR^+ \xi$  tends to zero, the pleating ray $\cal P_{\xi}$ approaches  the line $$ \Re \tau_i = 2 p_i(\xi)/q_i(\xi),  \Arg \tau_i = \pi/2, \Im \tau_1/\Im \tau_2 = q_2(\xi)/q_1(\xi).$$
\end{introthm}
      This theorem is stated more precisely  as Theorem~\ref{thm:intromaindetail}.  
  To actually locate the pleating rays,  note (Lemma~\ref{lemma:realtrace}) that $\tr \rho(\g)$ is real whenever $\g$ is a bending line. We prove:  
  \begin{introthm}
   \label{thm:intromaincompute}
    Suppose that $\xi   =  \sum_{1,2}a_i \delta_{\g_i}$ is admissible and  that the pair  $\g_1, \g_2$ is not exceptional.  Then any point on the ray $\P_{\xi}$ satisfies the equations $\Im \tr \rho(\g_i) =0, i=1,2$  and these  equations
have a unique solution as $\tau_i \to \infty$ in the direction specified by Theorem~\ref{thm:intromain}.
If the curve $\g_1$ is admissible,  then there exists $\g_2$ disjoint from $\g_1$ such that the pair $\g_1, \g_2$ is not exceptional, and thus $\P_{\g_1}$ is determined by the above result applied to $\xi   =   1 \cdot \delta_{\g_1} + 0 \cdot \delta_{\g_2}$.
  \end{introthm}

In the exceptional case we obtain only partial results detailed in Theorem~\ref{thm:exceptional}. We believe the above theorems to be still true in this case, but as discussed in  Section~\ref{sec:exceptional}, the result appears to be beyond the scope of this paper.
The lack of a complete result will not  affect the plotting of the asymptotic arrangement of pleating rays and planes. 
  
  As explained in Section~\ref{sec:examples}, these results  in principle enable  one to compute $\M$, modulo the unproven conjecture that the rational pleating rays are dense.
We hope to explore how to actually implement the computations in practice elsewhere.
  \medskip
  
 Theorems~\ref{thm:intromain} and~\ref{thm:intromaincompute} are proved together. The proofs have two main parts. First (Section~\ref{sec:behaviour})  we   show that asymptotically, the 
 lengths of the geodesic representatives  $\s^+_1,\s^+_2$ of  $\s_1,\s_2$ on $\dd \C^+/G$ tend to $0$ while at the same time becoming orthogonal to the bending lines. (This should be compared to the situation in~\cite{smallbend}, where in the limit as the bending  angles go to zero, the  bending lines on $\dd \C^+/G$ and $\dd \C^-/G$  become `orthogonal' in the sense that  average of the cosine of the angle between them goes to zero.)
 From this we deduce (Theorem~\ref{thm:intromain1}) that as $\th \to 0$, $\tau_1,\tau_2 \to \infty$ in such a way that 
 $$  \Arg \tau_i \to \pi/2, \Im \tau_1/\Im \tau_2 \to q_2(\xi)/q_1(\xi).$$

Second,    we use a formula for trace polynomials from~\cite{kps}. Note that 
 the trace   $\tr \rho(\g)$  is a polynomial on the parameter space $\CC^2$. The  formula,  see Theorem~\ref{thm:topterms}, expresses the  top terms of this polynomial  in terms of its  canonical coordinates $\i(\g)$.  We also make  use Thurston's symplectic form on the space of measured laminations $\ML$, which  turns out to have the standard  form relative to our canonical coordinates  (Section~\ref{sec:symplform}). To complete the proofs of Theorems~\ref{thm:intromain} and ~\ref{thm:intromaincompute},  
in Section~\ref{sec:asympdirns} we use the asymptotics of the trace polynomials together with Thurston's   form and some simple linear algebra to extract the unique possible asymptotic directions of the pleating rays.  

\medskip
One might also ask for the limit of the hyperbolic structure on $\dd \C^+/G$ as the bending measure tends to zero. The following result is an immediate consequence of the first part of the proof of Theorem~\ref{thm:intromain}:
 \begin{introthm}
   \label{thm:cnvgtopml}
Let $\xi   =  \sum_{1,2}a_i \delta_{\g_i}$  be as above.
Then as the bending measure $\beta(G)  \in \RR^+ \xi$  tends to zero, the induced hyperbolic structure of $\dd \C^+/G$ along   $\cal P_{\xi}$  converges to the barycentre of the laminations  $   {\s_1}$ and $ \s_2$  in the Thurston boundary of $\teich(\Sigma)$. 
 \end{introthm}
 This should be compared with the result in~\cite{smallbend}, that the analogous limit through groups whose bending laminations on the two sides of the convex hull boundary are in the classes of a fixed pair of laminations $[\xi^{\pm}]$, is a Fuchsian group on the line of minima of $[\xi^{\pm}]$. 

 \medskip

Although this paper is written in the context of the twice punctured torus,
 the results of Section~\ref{sec:behaviour} and hence also   Theorem~\ref{thm:cnvgtopml} should apply to the Maskit embedding of  a general surface. The top terms formula is needed only to determine the asymptotic value of $\Re \tau_i$.  Y. Chiang obtained an analogous top terms formula for the five times punctured sphere in~\cite{chiang}. We believe there is a  more general result  and hope to explore this elsewhere.

 \medskip
 
The plan of the paper is as follows.
 In Section~\ref{sec:maskit} we describe our holomorphic family of groups which realise the Maskit embedding and give estimates on the rough shape of $\M$.   In Section~\ref{sec:pleating} we  briefly review   facts about convex hull boundaries, bending measures  and pleating rays.  In Section~\ref{sec:simplecurves} we
 review canonical  coordinates for simple curves and the top terms formula from~\cite{kps}, and discuss  Thurston's symplectic form.  In Section~\ref{sec:examples} we  discuss in more detail how Theorem~\ref{thm:intromain} may be used to compute pleating rays and illustrate the theorem with  some very simple examples which can be computed by hand. The remaining two sections contain the main work of the paper as described above.  Theorem~\ref{thm:cnvgtopml}  is proved at the end of Section~\ref{sec:behaviour} and Theorems~\ref{thm:intromain} and~\ref{thm:intromaincompute}
 are proved in Section~\ref{sec:asympdirns}.
\bigskip
 
\noindent {\textbf{Acknowledgments}}
\medskip

This paper was begun in the early 1990's as joint project with Linda Keen and John Parker.  I would  like to thank them for allowing me to include some of our preliminary results and to continue alone. 
Most of our joint results are contained in~\cite{kps}; other work done   in draft only  we refer to here as~\cite{kps1}. We conjectured a partial version of Theorem~\ref{thm:intromain} but proofs were incomplete, in particular we lacked the orthogonality idea, the use of Thurston's symplectic form,  the use of Minsky's twist and the major general results in~\cite{BonO, cs}. 

I would like to thank MSRI and the organisers of the program on \emph{Teichm\"uller space and Kleinian groups} for their hospitality: during the program the research for this paper was completed.

\section{The Maskit embedding and plumbing parameters}
\label{sec:maskit}

Let $\Sigma$ be a twice punctured torus.   Figure~\ref{fig:funddomain} shows a fundamental domain $\Delta$ for a Fuchsian representation of $\Sigma$, on which some definite hyperbolic metric has been fixed.  The sides of $\Delta$ are identified by hyperbolic isometries $S_1,S_2, T$  which we can view as free generators for $ \pi_1(\Sigma)$.

 \begin{figure}[hbt] 
 \centering
\includegraphics[width=7cm,  viewport = 100 400  400  675]{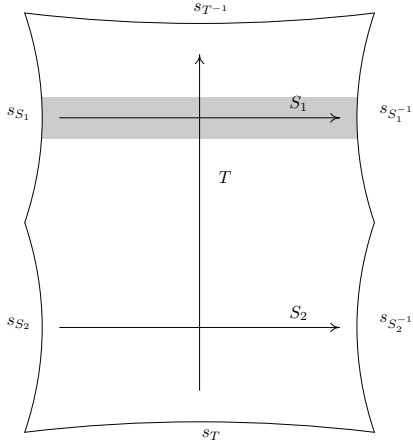}
\caption{The fundamental domain $\Delta$}
 \label{fig:funddomain}
\end{figure} 

\subsection{The Maskit embedding  }
\label{sec:maskitembed}

Let  $  \R(\Sigma) $ be the representation variety of $ \pi_1(\Sigma)$, that is, the set of 
representations $\rho(\pi_1(\Sigma))
\rightarrow SL(2,\CC)$ modulo conjugation in $SL(2,\CC)$ (but see e.g.~\cite{Kap} p.61) with the algebraic topology.
Let $\M \subset \R$ be  the subset of representations $\rho$ for which 
\begin{enumerate}
\renewcommand{\labelenumi}{(\roman{enumi})}
\item The image $G= \rho(\pi_1(\Sigma))$ is free and discrete and the images of $S_i$, $i=1,2$ are parabolic.
\item All components of the regular set $\Omega(G)$ are
simply connected and there is exactly one invariant component  $\Omega^+(G)$.
\item The quotient $\Omega(G)/G$ has $3$ components; 
$\Omega^+(G)/G$ is homeomorphic to $\Sigma$ and the other two components are 
  triply punctured spheres.
\end{enumerate}
 In this situation, see for example~\cite{Mar} Section 3.8,  the corresponding $3$-manifold $M_{\rho} = \HH^3/G$ is topologically $\Sigma \times (0,1)$.
Such a group $G$ is a  geometrically finite cusp group  on the boundary (in the algebraic topology) of the set of quasifuchsian representations of $\pi_1(\Sigma)$. The `top' component $\Omega^+/G $ of the conformal boundary    may be 
identified to  $\Sigma \times \{1\}$ and is homeomorphic to $\Sigma$. On the 
`bottom' component  $\Omega^-/G $, identified to  $\Sigma \times \{0\}$,  the two curves $\s_1,\s_2$ corresponding to the generators $S_1,S_2$ have been pinched, making  $\Omega^-/ G$  a union of two triply punctured spheres. The conformal structure on $\Omega^+/G $, together with the pinched curves  $\s_1,\s_2$, are the \emph{end invariants} of $M$ in the sense of Minsky's ending lamination theorem.
Since a triply punctured sphere is rigid, the conformal structure on  $\Omega^-/ G$ is fixed independent of $\rho$. The structure on  $\Omega^+/ G$  varies; 
it follows from standard Ahlfors-Bers theory  using the measurable Riemann mapping theorem  (see again~\cite{Mar} Section 3.8), that there is a unique group corresponding to each possible conformal structure  on $\Omega^+/ G$. 
Formally, the \emph{Maskit embedding} of the \Teich space of $\Sigma$ is the map
$$\Phi: \teich(\Sigma) \to \CC^2$$ which sends a point  $X \in \teich(\Sigma)$ to the 
unique group $G  \in \M$ for which $\Omega^+/ G$ has the marked conformal structure $X$.

\subsection{A concrete realisation of $\M$}
\label{sec:concrete}
Groups in $\cal M$ may be manufactured by the plumbing construction of Kra~\cite{kra},  see Section~\ref{sec:plumbing} below.  Here we simply write down a suitable holomorphic family of representations and verify directly   that groups thus constructed have the required properties. Groups in the family  depend on two complex parameters $\tau_1,\tau_2 \in \CC$. 

To define $\rho(\pi_1(\Sigma))
\rightarrow SL(2,\CC)$, it  suffices to give the images of the three free generators $S_1,S_2, T$  of  $ \pi_1(\Sigma)$. Following~\cite{kps}, for $\tau_1,\tau_2 \in \CC$ define $ \rho = \rho(\tau_1,\tau_2)$  by:
$$
    \rho (S_1) = \left(  \begin{array}{cc}
                 1  &  2 \\
                 0  &  1
              \end{array}    \right), 
\quad
    \rho (S_2) = \left(  \begin{array}{cc}
                 1  &  0 \\
                 2  &  1
              \end{array}    \right),
\quad
  \rho (T)= \left( \begin{array}{cc}
1 + \tau_1\tau_2 &  \tau_1 \\ \tau_2 & 1 \end{array} \right).
$$
Denote the image of $\rho(\tau_1,\tau_2)$ by $G(\tau_1,\tau_2)$. 
Note that  the holomorphic family $\G =  \{ G(\tau_1,\tau_2): \tau_i \in \CC \}$ has complex dimension $2$, the dimension of $ \TS$.  Not all groups $\G$ lie in $\M$, in particular any representation with $\Im \tau_i =0, i=1,2$ is Fuchsian and so not in $\M$.
We also need to restrict $\tau_1,\tau_2$ so as to have only one copy of each group up to conjugation. By abuse of notation, for $W \in \pi_1(\Sigma)$, we shall use $W$ also to denote the image $\rho(W)  \in G(\tau_1,\tau_2)$, and write $W = W(\tau_1,\tau_2)$ as needed to avoid confusion. 
  By direct computation (see Appendix $1$) we find   $\tr {[S_i,T^{-1}]} = \tau_j^2 + 2$ where $ j =1+i \  \mod 2$. 
Thus $\tau^2_1,\tau^2_2$ are invariant functions on $\R$, so that 
the component  $ \CC_+^2$ of $ \CC^2 \setminus  \Im \tau_i =0$ with  $\Im \tau_i >0, i=1,2$,  consists entirely of non-conjugate groups.  

The following propositions from~\cite{kps1} justify our use of the family $ \G$.
 
 \begin{proposition}
\label{thm:inM}
Let $G(\tau_1,\tau_2) \in \G$ be as above. If  $\Im \tau_i > 1$, $i=1,2$ and
$\Im \tau_1\Im \tau_2 > 4$ then $G(\tau_1,\tau_2) \in \M$.   Moreover the limit set $\Lambda(G)$ is contained in the two strips $0 \leq \Im z \leq 1/2$, $\Im  \tau_1-1/2 \leq \Im z \leq \Im  \tau_1$, together with the point at $\infty$.   \end{proposition}
\begin{proof} A fundamental domain for $G = G(\tau_1,\tau_2)$ is shown in Figure~\ref{fig:cxfunddomain}, in which the disks $B_2,B_3$ have equal  radius $ 1/\Im \tau_2$. The formal proof that $G$ is free and discrete is a straightforward application of Maskit's
second combination theorem, \cite{Maskitbook}  p.160.
 To see that $G\in \M$, one checks from the proof of the combination theorem that  the lower half plane and the half plane above the line $\Im z = \Im \tau_1$ project to 
the two triply punctured spheres which together form  $\Omega^-/ G$, while 
the simply connected component $\Omega^+ $ is contained in the strip $0 < \Im z <  \Im \tau_1$. The claim about  $\Lambda$  also follows. For further details, see Appendix $1$. \end{proof}

 \begin{figure}[hbt] 
\includegraphics[width=10cm, viewport = 5cm 16.5cm 17cm 24cm ]{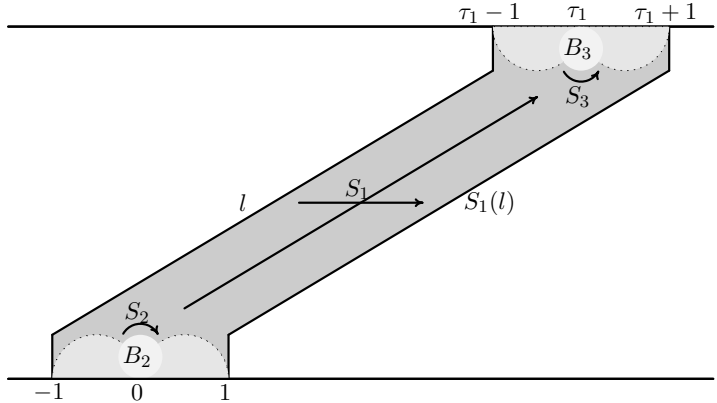}
\caption{A fundamental domain for $G(\tau_1,\tau_2)$.}
 \label{fig:cxfunddomain}
\end{figure} 

As explained above, a standard argument using the measurable Riemann mapping theorem now shows that \emph{any} group in $\M$ can be represented by a group in $
\G $. 
 Complementing Proposition~\ref{thm:inM} we have the following  result, also proved in  Appendix $1$.
\begin{proposition}
\label{thm:inM1}
Suppose that $G(\tau_1,\tau_2) \in \M$.
Then $\Im \tau_i \geq 1/2$, $i=1,2$ and $\Im \tau_1\Im \tau_2  \geq 1  $.
\end{proposition}

    We shall mainly do our calculations with $G$ normalised as above, so that $S_1(z) = z+2$. On occasion it is convenient to normalise $S_2$ in this way,  in which case we interchange the roles of $\tau_1$ and $ \tau_2$. More precisely we have:
\begin{lemma}
\label{lem:symmetry} The involution $K(z) = -1/{\bar z}$ conjugates $S_1$ to $S_2$. Moreover 
$KT(\tau_1, \tau_2) K^{-1} = T(-\bar \tau_2,-\bar \tau_1)^{-1}$.
\end{lemma}

To get some further feeling for $\M$, note that  the right Dehn twist $ D_{\s_1}$ around $\s_1$  induces the automorphism  $S_1 \mapsto S_1, S_2 \mapsto S_2$ and $T \mapsto S_1T$ of $\pi_1(\Sigma)$.  Since
$S_1T(\tau_1,\tau_2) = T(\tau_1+2,\tau_2)$, it follows that $ D_{\s_1}$ induces the map $(\tau_1,\tau_2) \mapsto  (\tau_1+2,\tau_2)$ on $\M$, and similarly for $ D_{\s_2}$.

 By abuse of notation, from now on we use $\M$ to denote the set of $(\t_1,\t_2)$ with  $\Im \tau_i >0$ and for which the group  $G(\t_1,\t_2)$ has properties (i), (ii) and (iii) above.
 The  above two propositions give rough bounds on the shape of $\M$.
 
\subsection{The plumbing construction}
\label{sec:plumbing}
The parameters $\tau_i$ have geometrical meaning as the \emph{plumbing parameters} of  Kra~\cite{kra}.    The idea is to make a projective structure on $\Sigma$ by starting from two triply punctured spheres with punctures identified in pairs. One `plumbs' across the punctures by  identifying punctured disk neighbourhoods $D,D'$ of the two paired cusps
  using the formula $zw=\tau$ where $z,w$ are holomorphic parameters in $D$ and $D'$. In the hyperbolic metric, this corresponds to identifying the punctured disks by  twisting by $\Re \tau$ and scaling by  a factor of $\Im \tau$.  
Making this construction in the above setting, with plumbing parameter $\tau_i$ across the puncture corresponding to $\s_i$,  results in the family $\G$.

\section{Bending lines and pleating rays}
\label{sec:pleating}

Let $M = \HH^3/G$ be a hyperbolic $3$-manifold, and let $\C/G$ be its convex core, where $\C$ is the convex hull in $\HH^3$ of the limit set of $G$, see~\cite{EpM}. If $M$ is geometrically finite then
there is a natural homeomorphism between each component  of $\dd \C/G$ and of $\Omega/G$. Each component $F$ of $\dd \C/G$ inherits  an induced hyperbolic structure from $M$. Moreover $F$ is a \emph{pleated surface}, meaning that it  is the isometric image under a map $f: H \to M$ of a hyperbolic surface $H$  which restricts to an isometry into $M$ on the leaves of a geodesic lamination $L$ on $H$, and which is also an isometry on each complementary component of $L$.  (Strictly, the pleated surface is the pair $(H,f)$; this becomes important when considering  the induced marking on $F$.)   We call  $L$  the \emph{bending lamination} of $F$ and the images of the complementary components of $L$, the \emph{flat pieces} of $F$. The bending lamination carries a  transverse measure called the \emph{bending measure}  which describes 
 the angles between the flat pieces, see~\cite{EpM} for details. 
By a  \emph{bending line} of $F$, we will mean any complete geodesic on $\partial \C/G$ which is either completely contained in $L$, or in the interior of a flat part.  We also use the term \emph{bending line} to mean any lift of a bending line of $F$ to a complete  geodesic  in $\HH^3$.
 
 We shall be  interested in manifolds for which the bending lamination  is \emph{rational}, that is, supported on closed curves.
Denote the 
 space of homotopy classes of simple closed essential loops on  $F$  by $\S(F)$ and the space of measured laminations on $F$ by $\ML(F)$. The subset of rational laminations is denoted $\ML_{\QQ}$ and consists of measured laminations of the form
 $ \sum_i a_i \delta_{\g_i}$, abbreviated $ \sum_i a_i {\g_i}$,  where the curves $\g_i \in \S(F)$ are disjoint and non-homotopic, $a_i  \geq 0$, and  $  \delta_{\g_i}$ denotes the transverse measure which gives weight $1$ to each intersection with $\g_i$. 
If  $ \sum_i a_i {\g_i}$ is the bending measure of a pleated surface $F$, then $a_i$ is the  angle  between the 
flat pieces adjacent to $\g_i$,  also denoted $\theta_{\g_i}$. In particular,   $\theta_{\g_i} = 0$ if and only if the  flat pieces adjacent to $\g_i$ are in a common totally geodesic  subset of $\dd \C/G$, equivalently have lifts which lie in  the same hyperbolic plane in $\HH^3$.

We take the term pleated surface to include the case in which a closed leaf $\g$ of the bending lamination maps  to the fixed point of a rank one parabolic cusp of $M$. In this case, the image pleated surface is cut along $\g$  and thus may be disconnected. Moreover the bending angle between the   flat pieces adjacent  to $\g$ is $\pi$. This is because for a geometrically finite group, the punctures on the two components of $\partial \cal C/G$ are paired and the associated flat pieces lift to two tangent half planes in $\HH^3$, see e.g.~\cite{Mar} Chapter 3  for details.

The key results about the existence  of hyperbolic manifolds with prescribed bending laminations  are due to Bonahon and Otal.  We need the following special case of their main theorem:

\begin{theorem}[\cite{BonO} Theorem 1]
\label{thm:bo} Suppose that $M $ is  $3$-manifold homeomorphic to  $\Sigma  \times (0,1)$, and that $\xi^{\pm} = \sum_i a^{\pm}_i \g^{\pm}_i \in \ML_{\QQ}(\Sigma) $.  Then  there exists a geometrically finite group $G$ such that $M = \HH^3/G$ and such that   the bending measures on the two components $\dd \C^{\pm}/G$   of $\dd \C/G$ equal $\xi^{\pm}$ respectively,  if and only if $a^{\pm}_i \in (0,\pi]$ for all $i$
and $\{ \g^{\pm}_i, i=1,\ldots,n\} $ fill up $\Sigma$, equivalently if $i(\xi^+, \gamma) +  i(\xi^-,\gamma)>0$ for every $\gamma \in \S$. If such a structure exists, it is unique.
\end{theorem}

Specialising now to the case of interest to this paper, let  $\rho= \rho(\t_1,\t_2)$ be a representation $\pi_1(\Sigma) \to SL(2,\CC)$ in the family $\G$ from Section~\ref{sec:maskit}, and suppose that the image $G = G(\t_1,\t_2) \in \M$.
The boundary of  the convex core $ \cal C/G$ has three components, one $\partial \cal C^+/G$ facing $\Omega^+/ G$ and homeomorphic to $\Sigma$, and  two triply punctured spheres whose union we denote $\partial \cal C^-/G$. 
The induced hyperbolic structures on the two components of $\partial \cal C^-/G$ are rigid, while the structure on $\partial \cal C^+/G$ varies. We denote the bending lamination of $\partial \cal C^+/G$ by $  \b(G) \in \ML$.  Following the discussion above, we view $\partial \cal C^-/G$ as a single pleated surface with bending lamination $ \pi (\s_1+\s_2)$, indicating that  two triply punctured spheres are glued across the annuli whose core curves   $\s_1$ and $\s_2$ correspond  to the parabolics $S_i \in G$. 

\begin{corollary}
  \label{cor:exist} A lamination $\xi  \in \ML_{\QQ}(\Sigma_{1,2}) $ is the bending measure of a group  $G \in \M$ if and only if $i(\xi, \s_1) , i(\xi, \s_2)>0$.  If such a structure exists, it is unique.
\end{corollary}
We call $\xi \in \ML_{\QQ}(\Sigma_{1,2})$ for which $i(\xi, \s_1) , i(\xi, \s_2)>0$, \emph{admissible}.

\subsection{Pleating rays} 
\label{sec:pleatingrays}
Denote the set of projective measured laminations on $\Sigma_{1,2}$ by $ \PML  $ and the projective class of $\xi = a_1 \g_1 + a_2 \g_2 \in \ML$ by $[\xi]$.
The \emph{pleating ray}  $ \P = \P_{[\xi]}$ of   $\xi \in \ML$ is the
set of groups $ G \in \M$ for which $ \b(G) \in [\xi]$.  To simplify notation  we write  $  \P_{\xi}$ for $ \P_{[\xi]}$ and note that   $  \P_{\xi}$ depends only on the projective class of $\xi $, also that $  \P_{\xi}$ is non-empty if and only if $\xi$ is admissible. In particular, we write 
$ \P_{\g}$ for the ray $ \P_{[\delta_{\g}]}$. 
As $ \b(G)$ increases,  $  \P_{\xi}$ limits on the 
unique geometrically finite group $G_{\rm cusp}(\xi)$ in the algebraic closure $\overline{ \M}$ of $\M$ at which at least one of the support curves to $\xi$ is parabolic, equivalently  so that $ \b(G) = \th(a_1 \g_1 + a_2 \g_2)$
with $\max \{  \th a_1,  \th\a_2\} = \pi$. We write 
$\overline{  \P_{\xi} }= \P_{\xi} \cup G_{\rm cusp}(\xi)$. 

Likewise for disjoint   non-homotopic curves $\g_1,\g_2 \in \S$, we define the \emph{pleating plane} $  \P_{\g_1,\g_2}$ of $\g_1,\g_2$   to be the set  of groups $ G \in \M$ for which 
$ \b(G) = \sum a_i \g_i$ with $a_i > 0$. Thus $  \P_{\g_1,\g_2}$ is the union of the pleating rays 
 $ \P_{\xi}$ with $ \xi = \sum a_i \g_i$, $a_i > 0$. The rays $  \P_{\g_1}$, $\P_{\g_2}$
 are clearly contained in the boundary of $  \P_{\g_1,\g_2}$;
note   that $\g_i$ may not be admissible even though
$ \xi = \sum a_i \g_i$ is. We call planes for which  one or other of the support curves is not admissible \emph{degenerate}, as they do not contain the corresponding ray $  \P_{\g_i}$.  
We write $\overline{ \P}_{\g_1,\g_2} = \cup_{\eta} \overline  \P_{\eta} $ where the union is over $ \eta = \sum a_i \g_i$ with $a_1,a_2 \geq 0$.

The following key lemma is proved in~\cite{cs} Proposition 4.1, see also~\cite{ksqf} Lemma 4.6. The essence is that because the two flat pieces of $\dd \C/G$ on either side of a bending line are invariant under translation along the bending line, the translation can have no  rotational part.
\begin{lemma}
  \label{lemma:realtrace} If the axis of $g \in G$ is a bending line   of $\dd \C/G$, then
$\tr(g) \in \RR$.
\end{lemma}
 Notice that the lemma applies even when the bending angle $\th_{\g}$ along $\g$ vanishes. Thus if $ G \in  {{\overline \P}_{\g_1,\g_2}}$ we have $\tr g \in \RR, i=1,2$,  for any $g  \in G$ whose axis projects either curve $\g_i$.

In order  to compute  pleating planes, we need the following result which is a  special case of Theorems B and C of~\cite{cs}, see also~\cite{kstop}. Recall  that a
 codimension-$p$ submanifold $N \hookrightarrow \CC^n$ is called \emph{totally real} if it is defined locally by equations $\Im f_i = 0, i =1,\ldots,p$, where $f_i, i = 1, \ldots,n$ are local holomorphic coordinates for $\CC^n$. 
As usual, if $\g$ is a bending line  we denote its 
 bending angle by $\theta_{\g}$. Recall that the \emph{complex length} $cl(A)$ of a loxodromic element $A \in SL(2,\CC)$  is defined by $   \tr A  = 2 \arc \cosh cl(A)/2$, see e.g.~\cite{swolpert} or ~\cite{cs} for details. By construction, $\P_{\g_1,\g_2} \subset \M \subset \R(\Sigma)$.
 
 \begin{theorem}
 \label{thm:cscoords}  The  complex lengths $cl \gamma_1,cl \g_2$ are local holomorphic coordinates for $  \R(\Sigma)$ in a neighbourhood of $\P_{\gamma_1, \g_2} $.
Moreover   $\P_{\gamma_1, \g_2} $ is connected and is locally defined as the totally real submanifold $\Im  \tr \gamma_i  = 0, i=1,2$ of $\R$. 
Any pair $(f_1,f_2)$, where $f_i$ is either the hyperbolic length  $ \Re cl(\g_i) $ or the bending angle  $ \theta_{\g_i}$, are global coordinates on $\P_{\gamma_1, \g_2} $.\end{theorem}

This result extends to $\overline{\P}_{\gamma_1, \g_2} $, except that  one has to replace $ \Re cl(\g_i) $  by  $\tr  \gamma_i$ in a neighbourhood of a point for which 
$\gamma_i$ is parabolic. 
In fact as discussed  in~\cite{cs} Section 3.1, complex length and traces are interchangeable except at cusps (where traces must be used) and points where a bending  angle vanishes (where complex length must be used). The parameterisation by lengths or angles extends to $\overline {\P}_{\gamma_1, \g_2} $.

 Notice that   the above theorem gives a local  characterisation $\overline {\P}_{\gamma_1, \g_2} $ \emph{as a subset of the representation variety $\R$} and not just of $\M$.  \emph{In other words, to locate $\P$, one does not need to check whether nearby points lie \emph{a priori} in $\M$; it is enough to check that the traces remain real and away from $2$ and that the bending angle on one or other of $\theta_{\g_i}$ does not vanish.} As we shall see, this last condition can easily be checked by requiring that further  traces be real valued.

\section{Canonical coordinates for simple curves}
\label{sec:simplecurves}

Our main result Theorem~\ref{thm:intromain} involves the explicit  coordinatisation of the space $\ML = \ML (\Sigma_{1,2})$ of measured laminations on $\Sigma_{1,2}$ introduced in~\cite{kps}. 
The coordinates, called  $\pi_{1,2}$-coordinates  in~\cite{kps} and  \emph{canonical coordinates}  in this paper, are essentially  Dehn-Thurston  coordinates 
relative to the curves $\s_1,\s_2$. They are global coordinates for $\ML$ which take values in $(\RR_+ \times \RR)^2$, the sign of a given coordinate in a given chart being constant. At the same time, they can be viewed as giving 
a piecewise linear cone structure to $\ML$,  the charts being laminations supported on a particular set of train tracks on $\Sigma$.  Modulo their boundaries, these charts partition $\ML$; which chart supports a given lamination being determined by simple linear inequalities between their coordinates.
The coordinates are set up so as to maintain as close an analogy as possible  with the once punctured torus $\Sigma_{1,1}$.  
The charts are also very closely related to the $\pi_1$-train tracks of~\cite{birmans}.

In more detail, the canonical coordinates 
$${\bf i}(\gamma)=(q_1(\g) ,p_1(\g),q_2(\g),p_2(\g)) \in (\ZZ_+ \times \ZZ)^2$$ of a curve $\g \in \S( 
\Sigma_{1,2})$ are defined as follows.  We set  $q_i(\g)  = i(\g,\s_i)  \geq 0$.
 Since  $\s_1,\s_2$ together bound a pair of pants we note:
\begin{equation}
\label{eq:qcong}
q_1+q_2 \cong 0 \mod 2.
\end{equation}
This equation will be important in Section~\ref{sec:asympdirns}.

The definition of $p_i(\g)$ (which is more complicated and can be omitted at first reading) is made
relative to the  fundamental domain $\Delta$ of Section~\ref{sec:concrete}.  The reader may find the  discussion for $\Sigma_{1,1}$ in Appendix $2$ enlightening.

 \begin{figure}[hbt] 
 \centering
\includegraphics[width=8cm, viewport =100 300  400  675 ]{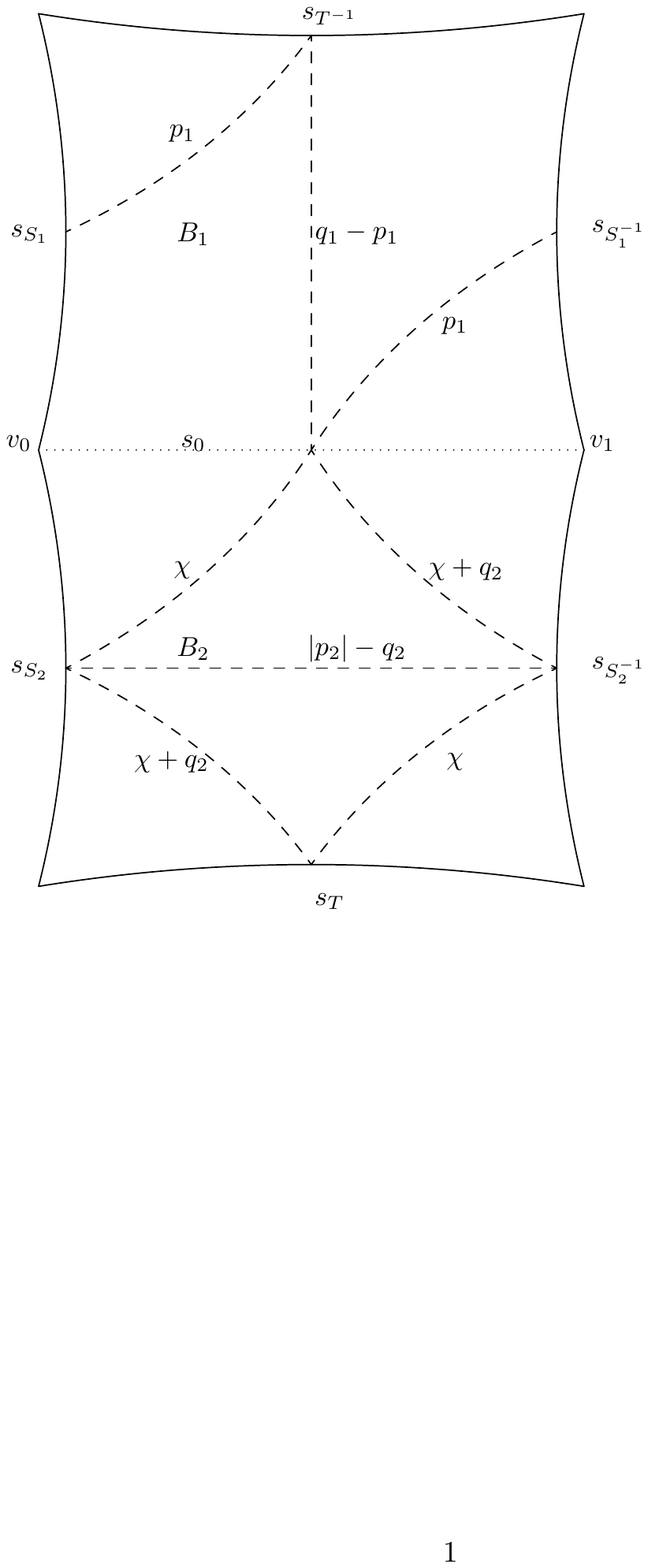}
\caption{Typical configuration for canonical coordinates. The cell shown is defined by the five inequalities $ p_1 \geq 0, p_2 \leq 0, q_1 \geq p_1,  q_2 \leq -p_2, q_2 \leq q_1$.}
 \label{fig:basiccanon}
\end{figure} 

Referring to  Figure~\ref{fig:basiccanon}, label each side of $\Delta$  by the  generator  which carries it  to a paired side, so that  the bottom  side is labelled $s_T$  since $T$ carries it to the top side $s_T^{-1}$ ;   similarly the top  left side  is labelled  $s_{S_1}$ since $S_1$ carries it to the right side $s_{S_1^{-1}}$;
 and the lower left and right sides are labelled $s_{S_2}, s_{S_2^{-1}}$ respectively.    Since each side joins a puncture to a puncture, the intersection numbers of a curve $\g \in \S$ with these sides are well defined.
Let  $s_0$ be the arc joining the vertices $v_0,v_1$  in the middle of the vertical sides of $\Delta$, which both project to the puncture $S_2S_1^{-1}$ on $\Sigma_{1,2}$.  This partitions
$\Delta$ into two four sided `boxes' $B_1, B_2$ with sides $s_0, s_{S_i},  s_{{S_i}^{-1}} $ and $s_{T^{-1}}$ or $s_{T}$ respectively.  

The geodesic representative of $\g$ on $\Delta$ intersects each box $B_i$ in a number of pairwise disjoint arcs, none of which runs from one side of $B_i$ to itself.
We observe, see~\cite{kps} Lemma 4.1,  that in at least one box $B_i$, there is either no `corner arc' joining 
 $s_0$ to $ s_{S_i}$, or no corner arc  joining   $s_0$ to $ s_{{S_i}^{-1}} $.
This is because if strands of $\g$ included all four corner arcs in both boxes, the `innermost' such arcs  would link  to form  a loop  round  the puncture  $S_2S_1^{-1}$
which is the projection  to $\Sigma_{1,2}$ of $v_0$ and $v_1$.  This is impossible. 
In a similar way, there cannot be corner arcs surrounding all of the four vertices 
 $w_0, \ldots, w_3$  marked in  Figure~\ref{fig:basiccanon}, since their innermost strands would link to form a loop round the  common projection of the $w_i$, the puncture $S_1TS_2^{-1}T^{-1}$.
 Using this together with the  `switch conditions' 
$  i(\g,s_{S_i})=i(\g,s_{S_i^{-1}})$ and $  i(\g,s_{T})=i(\g,s_{T^{-1}})$, one checks (see~\cite{kps} Section 4)  that $i(s_0,\g) = i(s_0,s_{T})=i(s_0,s_{T^{-1}})$. 
It also follows that  there are equal number of  
 arcs joining each pair of diagonally opposite corners of each $B_i$.

Suppose for definiteness that the box with missing corner arcs is $B_1$. 
 In this situation, it is not hard to see that   $q_1(\g)  = i(\g,{\s_1})  = i(\g,{\s_0})$.
In analogy with the case of $\Sigma_{1,1}$ as described in Appendix $2$,  we define  $|p_1| = i(\g, s_{S_1}) =i(\g, s_{S_1^{-1}})$. Since $ i(\g, s_0) \geq i(\g, \s_2) =q_2(\g)$, we have $q_2 \geq q_1$.
 Set $ \chi = |q_2-q_1|/2$. (Note that $ \chi $ is integral by~\eqref{eq:qcong}.) 
One   verifies that
$q_2(\g)  = i(\g,s_0)  - 2 \chi  \geq 0$ and we define $|p_2| = i(\g, s_{S_2}) - 2 \chi =i(\g, s_{S^{-1}_2}) - 2 \chi $.  If $B_2$ is the box with missing corner arcs,  make similar definitions with the roles of $p_1,p_2$ reversed.

To fix the sign of $p_1$, take $p_1 > 0$ if there is an arc of $\g$ joining $s_{S_1} $ to $s_{T^{-1}}$ and $p_1 \leq 0$  otherwise. 
Likewise take $p_2> 0$ if there are $w > \chi$ arcs of $\g$ joining  $s_{S_2}$ to $s_0$   and $p_2 \leq 0$  otherwise.  If $q_2 \leq q_1$ we make similar definitions interchanging the indices $1,2$.

\smallskip
The intuition behind this definition  is that we are thinking of each $B_i$ as  a fundamental domain  for  a once punctured torus $\Sigma_{1,1}^i$ with  opposite sides  identified. With such an identification, the arcs of $\g$  in $B_i$   would glue up to form a multiple loop on $\Sigma_{1,1}^i$.  
As is well known, closed curves on $\Sigma_{1,1}$ are in bijective correspondence with lines of rational slope  in $\RR^2$.
  The coordinate  $(q_i,p_i)$ indicates that the  `slope' of $\g$ in  $B_i$ is $p_i/q_i$, see Appendix $2$. 
  The number $\chi$ is the number of `corner strands'  joining adjacent sides of $B_i$ which, were the box actually  $\Sigma_{1,1}$, would link to form $\chi$ copies of a loop round the puncture. Such corner strands  occur in  box $B_2$ if and only if   $q_2>q_1$; this is why we subtracted $\chi$ from the $B_2$-intersection numbers only. Note that, in contrast to $\Sigma_{1,1}$, the numbers $q_i,p_i$ are no longer in general relatively prime.   Up to the choice of a base point for the twist, $p_i/q_i$ is the Dehn-Thurston twist coordinate of $\g$ relative to $\s_i$, see~\cite{minskyproduct} Lemma 3.5 and Section~\ref{sec:twist} below.

  \smallskip

The connection between  this definition and the more standard cell decomposition of $\ML$ by weighted train tracks is as follows. Collapse all the arcs of $\g$ joining one side  of $B_i$ to another, into a single strand joining the midpoints of the same  two sides. 
The collapsed strands join across $s_0$ to form a train track on $\Sigma_{1,2}$, 
whose branches are the strands and all of whose switches are at the midpoints of the sides (including $s_0$).
In~\cite{kps} we called these special tracks, $\pi_{1,2}$-train tracks; here we refer to them as \emph{canonical}.  The non-negative weights on this collection of train tracks form a cell decomposition of 
$\ML(\Sigma_{1,2})$.  
However the coordinates $(|p_i|,q_i)$ are \emph{not} in general equal to   the actual  weights on  these rather complicated configurations of branches.   Rather, $\i(\g)$ has global meaning.
Inequalities among the coefficients of  $\i(\g)$ determine which track in the cell decomposition supports a given curve $\g$ and, given the cell,  the coefficients determine the weights on the track. 

  \smallskip 
  
Canonical coordinates extend naturally  by linearity and continuity to global coordinates for $  \ML (\Sigma_{1,2})$.  
 The cell structure is best understood by referring to Figure~\ref{fig:opttraintracks} in Appendix $2$. This shows the  four cells for $\Sigma_{1,1}$ which glue to form $  \PML (\Sigma_{1,1})= S^1$. Each of the four configurations in Figure~\ref{fig:opttraintracks}  can occur in either $B_1$ or $B_2$, and for each such configuration there are $2$ further options for the box in which $q_i > q_j$, This makes in all $32$ cells, the images of which in $\PML$ glue along their faces to make $S^3$, see~\cite{kps}. 

\smallskip

An  important feature of canonical coordinates is that it is easy to read off the coordinates  of a curve $\g_{W} \in \S $ represented by a word $W$ in the generators $S_1,S_2,T$ of $\pi_1(\Sigma)$.  First, write $W$ as a cyclically shortest word  $e_1e_2 \ldots e_n$  and set $e_{n+1} = e_1$.  Draw  arcs on $\Delta$  from $s_{e_i}$ to $s_{e_{i+1}^{-1}}, i=1, \ldots, n$. 
Suppose that $\g_W$ is simple on $\Sigma$. Then by~\cite{kps} Theorem 3.1, these arcs can be arranged so as to be pairwise  disjoint and the weighted canonical track they define  gives precisely the canonical coordinates of $\g_W$. This method, similar to the method of \emph{ $\pi_1$-train tracks} developed at length in~\cite{birmans}, was crucial in the proof of the top terms formula below. Some examples are given in Section~\ref{sec:examples}.

\subsection{Top terms formula}
Canonical train tracks and  coordinates were used in~\cite{kps}   to study matrices $\rho(W)$ in the family $\G$ of Section~\ref{sec:concrete}, where $W \in \pi_1(\Sigma_{1,2})$ corresponds to $\g_W \in \S$. The matrix coefficients and hence the trace
 $\tr W$ are clearly polynomials in $\tau_1,\tau_2$.

\begin{theorem}[\cite{kps} Theorem 6.1]
\label{thm:topterms}
Let $\gamma$ be a simple closed curve on $\Sigma$ with canonical coordinates 
${\bf i}(\gamma)=(q_1,p_1,q_2,p_2)$.  Let $\gamma$ be
represented by $W \in \pi_1(\Sigma)$. Then if $q_1,q_2>0$:
$$\tr W=   \pm 2^{|q_2-q_1|}\bigl(\tau_1+2p_1/q_1\bigr)^{q_1}
\bigl(\tau_2+2p_2/q_2\bigr)^{q_2}
+R\bigl(q_1+q_2-2\bigr) $$
where  $R(q_1+q_2-2)$ denotes a polynomial  of degree at most $q_1$ in
$\tau_1$ 
and $q_2$ in $\tau_2$ and with total degree in $\tau_1$ and $\tau_2$ at most
$q_1+q_2-2$.
If $q_2 = 0$ then $\tr W$ is a polynomial in $\tau_1$ only, and $$\tr W=   \pm 2^{q_1}\bigl(\tau_1+2p_1/q_1\bigr)^{q_1}
+R\bigl(q_1-2\bigr), $$ while if $q_1=0$ there is a similar expression in $\tau_2$.
\end{theorem}

\subsection{The Thurston symplectic form}
\label{sec:symplform}
Thurston defined a symplectic form $\Omega_{\Th}$ on $ML$, the symplectic product being defined for curves carried by a common train track $\tau$, see for example~\cite{Penner}. 
By splitting, we can arrange that every switch $v$ of $\tau$ is trivalent with one incoming branch and two outgoing ones.  Since $\Sigma$ is oriented, we can distinguish the right and left hand outgoing branches, the left hand branch being the one to be followed by a British driver approaching $v$ from the incoming branch, see Figure~\ref{fig:switch}. If ${\bf n}, {\bf n'}$ are non-negative weightings on $\tau$ (representing points in $ML$), we denote by ${b_v}({\bf n}), {c_v}({\bf n})$
 the weights of the left hand and right hand outgoing branches at $v$ respectively. The Thurston product is defined as
 $\Omega_{\rm Th} ({\bf n}, {\bf n'}) = \frac{1}{2}\sum_{v}b_v({\bf n}) c_v({\bf n'}) -b_v({\bf n'}) c_v({\bf n})$.

 \begin{figure}[hbt] 
 \centering
\includegraphics  [height = .9 in, viewport = 180 600  500  675]{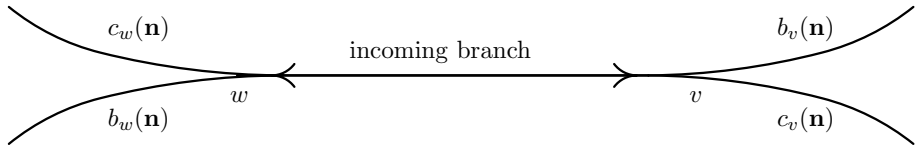} 
 \caption{Weighted branches at a switch.}
 \label{fig:switch}
\end{figure} 

Bearing in mind the interpretation of canonical coordinates as weights on train tracks,
we can  interpret this definition in terms of   canonical  coordinates.  
It is then not hard to check the following rather remarkable result:
\begin{proposition}
\label{prop:symplform}
Suppose that loops   $\gamma, \gamma' \in \S$ are supported on a common canonical  train track with coordinates 
${\bf i}(\gamma)=(q_1,p_1,q_2,p_2), {\bf i}(\gamma')=(q'_1,p'_1,q'_2,p'_2)$. 
Then $\Omega_{\rm Th} (\gamma,\gamma') = \sum_{i=1,2} (q_ip'_i- q'_ip_i).$
If $\g,\g'$ are disjoint, then  $\Omega_{\rm Th} (\gamma,\g') =0$.
\end{proposition}
To check the last statement, note first that disjoint curves are necessarily supported on the same canonical  track. Then check  that $\Omega_{\rm Th}$ is invariant under splitting and shifting. Then split and shift until the two curves are supported on disjoint train tracks.

\section{Computation and examples}
\label{sec:examples}
Before embarking on the proof of~Theorem~\ref{thm:intromain}, we briefly discuss its implications for computation and then give a few examples which it is possible to work out by hand.     
     
\subsection{Computation}
Let us return to the original problem of locating $\M$. 
Conjecturally, the pleating planes are dense in $\M$. (This was proved for $\M(\Sigma_{1,1})$ in~\cite{kstop}.) Thus we concentrate on the problem of locating a given pleating plane $\cal P_{\g_1,\g_2}$ in the parameter space $\HH^2 \subset \CC^2$.
Let $\cal V_{> 2}(\g_i), \cal V_{= 2}(\g_i)$ denote respectively the real analytic varieties in $\HH^2$ on which $\tr \g_i \in \RR\setminus [-2,2]$ 
and $\tr \g_i = \pm 2$.
By Theorem~\ref{thm:cscoords}, $\cal P_{\g_1,\g_2}$ is a connected totally real submanifold  of $\cal V_{> 2}(\g_1) \cap \cal V_{> 2}(\g_2)$. Its boundary is contained in $\cal V_{= 2}(\g_1) \cup \cal V_{= 2}(\g_2)$  and the two rays 
$  \P_{\g_1}$ and $ \P_{\g_2}$.
By Theorem~\ref{thm:intromaincompute}, as long as the pair $\g_1,\g_2$ is not exceptional (for which see  Lemma~\ref{lem:transverse} and Theorem~\ref{thm:exceptional}), $\cal V_{> 2}(\g_1) \cap \cal V_{> 2}(\g_2)$ is a $2$-manifold and has a unique branch 
near infinity satisfying the conditions of Theorem~\ref{thm:intromain}.

The ray  $  \P_{\g_1}$  may be  determined as follows.
At points on  $  \P_{\g_1}$, the surface $\dd \C^+/G$ cut along 
$ \g_1 $ is flat. Therefore the trace of \emph{any} curve $\d$ disjoint from 
$\g_1$  lies in the variety $\cal V_{> 2}(\d)$.
One can show (see  Lemma~\ref{lem:3indeppeople}) that of any two admissible curves $\d,\d'$ disjoint from $\g_1$
and distinct from $\g_2$, at least one will have the property that $\cal V_{> 2}(\d) \cup \cal V_{> 2}(\g_1)$ is transverse to $\cal P_{\g_1,\g_2}$, so that the equations $ \Im \g_1, \Im \d, \Im \d' \in \RR$ together with the conditions of 
 Theorem~\ref{thm:intromain}
 are enough to determine $  \P_{\g_1}$.
Finally, since $\cal P_{\g_1,\g_2}$ is  the union of the rays 
$\cal P_{\xi}$ for $\xi$ with support $\g_1,\g_2$, 
it must  the connected component of the unique branch of $\cal V_{> 2}(\g_1) \cap \cal V_{> 2}(\g_2)$ which interpolates between
$\cal P_{\g_1}$ and $\P_{\g_2}$ as $\tau_i \to \infty$.

This means that in principle, given a means of computing sufficiently many traces, $\cal P_{\g_1,\g_2}$ can be determined computationally
\emph{without needing any further tests to see if $G$ is discrete}. In particular, one can locate the one (real)-dimensional boundary of $\cal P_{\g_1,\g_2}$ on $\dd \M$, along which at least one of the elements $ \g_1,\g_2$ is parabolic.  By~\cite{mcmullen},  such geometrically finite group  
cusp groups are dense in  $\dd \M$.  In the one dimensional case, this is exactly the procedure which gives Figure~\ref{fig:maskit}.  We hope to discuss the practical implementation of this
programme  elsewhere.

\subsection{Examples} 
Following~\cite{kps1}, we will identify pleating rays and planes formed by various combinations of the  curves corresponding to the elements  $W = S_i, T, [S_i,T^{-1}], S_1S_2^{-1}T^{-1}$ and $S_1^{-1}S_2T^{-1}$
in $\pi_1(\Sigma)$.
Let $\g_{W}$ be  the curve corresponding to element 
 $W$  and write $\P_{W}$ for $\P_{ \g_{W}}$ and   $\tr W$ for $\tr \rho(W)$ 
 with $\rho = \rho(\tau_1, \tau_2) $ defined as in Section~\ref{sec:concrete}.

We compute:
$ \tr T = 2 + \tau_1\tau_2$, 
$\tr {[S_i,T^{-1}]} = 2+ 4\tau_j^2, j \neq i$ (see Appendix 1), and 
 $\tr  S_1 S_2^{-1} T^{-1} = (\tau_2-2)(\tau_1+2)$, where $\tr   S_1 S_2^{-1} T^{-1}$  can either be computed directly or deduced from the fact  that the right  Dehn twist $D_{\s_i}$ about $S_i$ induces the map  $\tau_i \mapsto \tau_i +2$, see Section~\ref{sec:concrete}. 
Writing the canonical coordinates in the usual form $${\bf i}(\xi)= (q_1(\xi),p_1(\xi),q_2(\xi),p_2(\xi)),$$ we find following the method outlined in Section~\ref{sec:simplecurves}:
\begin{eqnarray*}
{\bf i}(T) = (1,0,1,0), \    \  {\bf i}([S_1,T^{-1}]) = (0,0,2,0), \    \ {\bf i}([S_2,T^{-1}]) = (2,0,0,0),\\  {\bf i}( S_1S_2^{-1}T^{-1}) = (1,-1,1,1),  \   \ {\bf i}( S_1^{-1}S_2 T^{-1}) = (1,1,1,-1).  
\end{eqnarray*}
The  curves $ S_1S_2^{-1}T^{-1}$ and $S_1 T^{-1}S_1^{-1} T$ are illustrated in Figure~\ref{fig:cancurves}.

\begin{figure}[hbt] 
\centering 
\includegraphics [ viewport = 140 465  500  675] {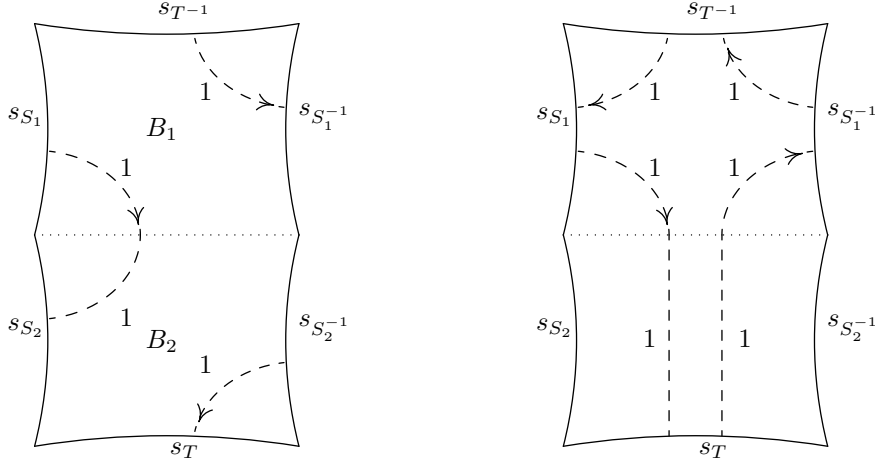} 
\caption{The curves $ S_1S_2^{-1}T^{-1}$ (left) and $S_1 T^{-1}S_1^{-1} T$ (right).}
\label{fig:cancurves}
\end{figure} 
\medskip

The pleating planes discussed in Examples 1--3a below are shown in the left frame of Figure~\ref{fig:pleatingplanes} and those of
Example 4 on the right.

\medskip

\noindent {\bf Example 1: The  pleating ray $\P_{T}$.} On $\P_{T}$, the surface $\Sigma_{T} := \Sigma \setminus \g_{T}$ is 
flat and hence not only $\g_{T}$, but also any curve contained in $\Sigma_{T}$, has real trace. 
Since  $\g_{T}$ is disjoint from 
both $\g_{[S_1,T^{-1}]}$ and $\g_{[S_2,T^{-1}]}$, it follows that $  2+ 4\tau_j^2 \in \RR, i=1,2$ and  
$2 + \tau_1\tau_2 \in \RR$. We deduce that on $\P_{T}$, $ \Re \tau_1 = \Re \tau_2 = 0$.
Since   $\g_{T}$ is also disjoint from $\g_{ S_1S_2^{-1}T^{-1}}$, by the same reasoning,  $\tau_1\tau_2 + 2\tau_2 - 2\tau_1 \in \RR$,  from which we deduce  $ \Im \tau_1 = \Im \tau_2$. (Disjointness of curves can be easily checked by drawing disjoint representatives on the fundamental domain $\Delta$, taking into account how the arcs link across the glued sides. In this instance,  the curve $T$ is an arc from $s_T$ to $s_{T^{-1}}$ which is clearly disjoint from both curves illustrated in Figure~\ref{fig:cancurves}.)

The  conditions $ \Re \tau_1 = \Re \tau_2 = 0,  \Im \tau_1 = \Im \tau_2$  define a line $L$ in $\CC_+^2$. Moreover $| \tr T| > 2$ on $L$ whenever  $\Im \tau_1 >2$, and $T$ is parabolic exactly at the point at which $\Im \tau_1=2$.
It follows from Theorem~\ref{thm:cscoords} and the discussion above, that  $\cal P_{T} = L \cap \{ \Im \tau_1>2\}$. To compare with Theorem~\ref{thm:intromain}, since $p_i(T)/q_i(T) = 0,i =1,2$ and $q_2(T)/q_1(T) = 1$, in this special case, $\cal P_{T}$ is actually \emph{equal}  to  the line $$ \Re \tau_i = 2 p_i(T)/q_i(T),  \Arg \tau_i = \pi/2, \Im \tau_1/\Im \tau_2 = q_2(T)/q_1(T).$$
The endpoint point $E_{T}=  (2i, 2i)$ of $\P_{T}$ represents the unique group $G_{\rm cusp} (T)$ for which $S_1,S_2$ and $T$ are all parabolic and the remaining part of $\dd \C^+/G$ is totally geodesic.

In this very special case, it is also possible to verify directly, following the methods of~\cite{kstop},  that groups on $L$ have convex hull boundary which is bent exactly along the curve $\g_T$, and that by symmetry the geodesic axis of 
$\g_T$ meets both the geodesic representatives of $\s_i = \g_{S_i}$ on $\dd \C^+/G$ orthogonally.
 \medskip
 
\begin{figure}[hbt] 
\centering
\begin{minipage}[t]{2.5cm}
 \includegraphics [height = 2.5in, viewport = 210 430  450  675] {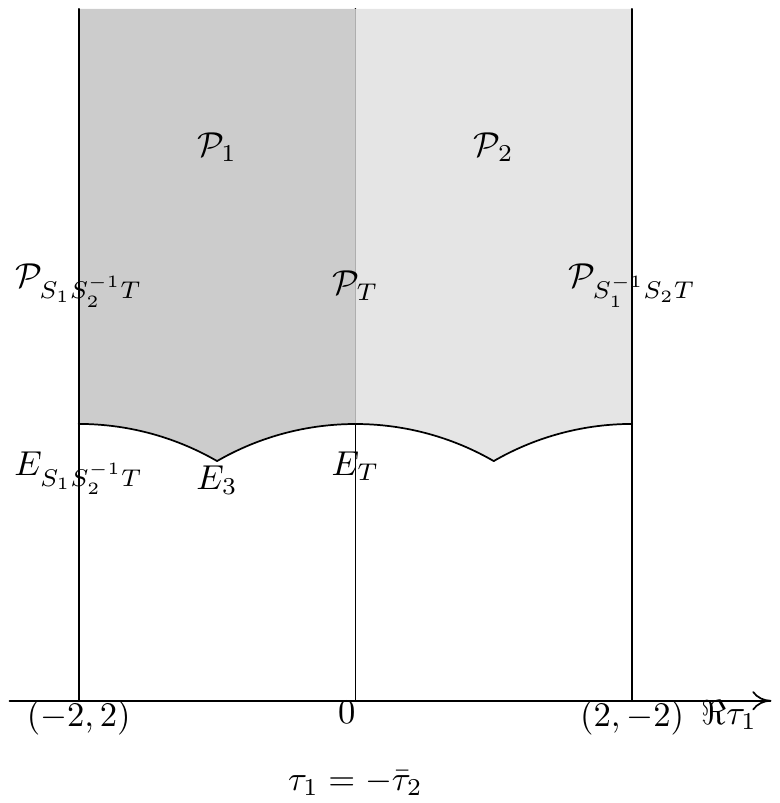} 
 \end{minipage}
 \hspace{4cm}
\begin{minipage}[t]{2.5 cm}
 \includegraphics [height = 2.5 in, viewport = 210 430  450  675] {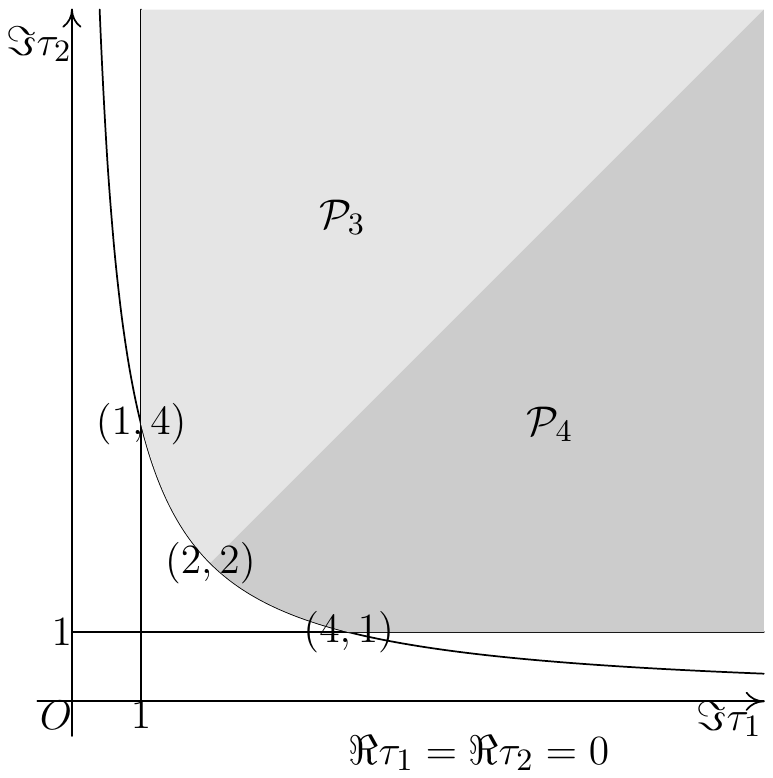} 
 \end{minipage}
\caption{Pleating planes. Left: Examples 1 -- 3a.  Right: Example 4.}
\label{fig:pleatingplanes}
\end{figure} 
    
   \noindent {\bf Example 2: The  pleating ray $\P_{ S_1S_2^{-1}T^{-1}}$.}  The easy way to locate this ray is to note that $S_1S_2^{-1}T^{-1} = D_{\s_1} D^{-1}_{\s_2} (T)$. Since $D_{\s_i}$ is a symmetry of $\M$, it follows immediately that $\P_{S_1S_2^{-1}T^{-1}}$ is the line
$ \Re \tau_1 = -2, \Re \tau_2 =2,  \Im \tau_1 = \Im \tau_2, \Im \tau_1 >2$, in other words, the line $L + 2(-1,1)$ where $L$ is as in Example 1.
Since in this case  $p_1/q_1 = -1, p_2/q_2 = 1$ and $q_2/q_1 = 1$,
this is again in accordance with Theorem~\ref{thm:intromain}.
 
The point $E_{ S_1S_2^{-1}T^{-1}} = (-2+ 2i, 2+2i) $  is the unique group for which $S_1,S_2$ and $S_1S_2^{-1}T^{-1}$ are all parabolic and the remaining part of $\dd \C^+/G$ is totally geodesic.
 \medskip
 
  \noindent {\bf Example 3: The  pleating plane  $\P_1=   \P_{T,  S_1S_2^{-1}T^{-1}}$.} 
  By Theorem~\ref{thm:cscoords} and the discussion above, $\P_1$ is  a connected  open subset of  the plane $\Pi $  defined by $  \tau_1 =  - \bar \tau_2$ whose  boundary   contains the lines $\P_{T}$ and $\P_{S_1S_2^{-1}T^{-1}}$.  To compute the remaining boundary of $ \P_1$, note that  the conditions  $| \tr T|, | \tr S_1S_2^{-1}T^{-1}| > 2$ imply
  $ |\tau_1|^2 > 4,\, |\tau_1 -2|^2 > 4$.  The  equations  $ |\tau_1|=2,\, |\tau_1 -2|=2$ define two circular arcs which meet at the point $E_3$ at which $\Re \tau_1 = -1, \Im \tau_1 = \sqrt 3$.
Then  $T$ is parabolic along the arc of   $ |\tau_1|= 2$ from $E_T$ to $E_3$ 
 and $ S_1S_2^{-1}T^{-1}$  is parabolic along the arc of   $ |\tau_1-2|= 2$ from $E_{S_1S_2^{-1}T^{-1}}$ to $E_3$.  
  At $E_3$, both $T$ and $S_1S_2^{-1}T^{-1}$
  are  parabolic. Thus $E_3$   represents a double cusp group in which  $T, S_1S_2^{-1}T^{-1}, S_1,S_2$ are all parabolic.  (There is a unique such group, either by the ending lamination theorem, or more simply by~\cite{kms}.) 
It follows from the above discussion that 
  $$\P_{T,S_1S_2^{-1}T^{-1}} = \{(\tau_1,\tau_2)| \,\tau_1 = -\bar\tau_2,\,
|\tau_1| > 2,\,  |\tau_1 -2| > 2,\, -2 <\Re \tau_1 < 0\}, $$ as illustrated in the left frame of Figure~\ref{fig:pleatingplanes}.
    
\medskip
The next example shows that distinct planes  may be contained in the same real trace variety in $\CC_+^2$. 
\medskip

\noindent {\bf Example 3a: The  pleating plane  $ \P_2 =\P_{T,  S_1^{-1}S_2T^{-1}}$.}   Similarly to Example 2, we compute that $\P_{
S_1^{-1}S_2T^{-1}}$
is the line  $L + 2(1,-1)$. 
The same reasoning as in Example 3 shows that 
$ \P_{T,  S_1^{-1}S_2T^{-1}}$
  is the region contained the same plane $\Pi$, but  
bounded by the lines $\P_{T}$ and $\P_{S_1^{-1}S_2T^{-1}}$ and 
above the arcs $ |\tau_1|= 2$ and $ |\tau_1+2|= 2$.
In other words 
$$\P_{T, S_1^{-1}S_2T^{-1}} = \{(\tau_1,\tau_2) | \,\tau_1 = -\bar\tau_2,\,
|\tau_1|  > 2,\,  |\tau_1 +2| > 2,\, 0 <\Re \tau_1 < 2\}.$$

 Thus the two pleating planes $\P_1$ and $\P_2$ are both contained same plane $\Pi \subset \CC_+^2$. They meet along the line $\P_{T}$ contained in the boundary of both, see the left frame of Figure~\ref{fig:pleatingplanes}.

\medskip
Our final example is of degenerate pleating planes.
\medskip

\noindent{\bf Example 4: The  pleating planes  $ \P_{T,  [S_i, T^{-1}]}$.}  
In Example 3, the curves $T$ and $S_1S_2^{-1}T^{-1}$ are themselves are admissible, so that the corresponding pleating rays are non-empty.
By contrast, since $q_1 = i( [S_1, T^{-1}],S_1) = 0$, the curve  $ [S_1, T^{-1}]$ by itself is not admissible so that $\P_{ [S_1, T^{-1}]} = \emptyset$.

From $\tr T = 2+ \tau_1\tau_2$, 
$\tr  {[S_1,T^{-1}] } = 2+ 4\tau_2^2 $ we find $ \P_{T,  [S_i, T^{-1}]}$
is contained in the plane  $\Pi'= \{(\tau_1,\tau_2) \in \CC_+^2:  \Re\tau_1=\Re\tau_2 = 0 \}$.
To locate $ \P_{T,  [S_i, T^{-1}]}$,  we locate its boundary curves in $\Pi'$.
Part of the boundary is the line $\P_T = L$ of Example 1, along which $\Im \tau_1 =\Im \tau_2$. 
In $\Pi'$, the element $[S_1,T^{-1}]$ is parabolic along the line $\tau_2 =i$ and  $T$ is parabolic along the hyperbola  $\{\Im \tau_1 \Im \tau_2 = 4\}$.
These two loci meet exactly once at the point $(4i,i)$ which therefore represents the maximally pinched group for which $S_1,S_2, T$ and $[S_i, T^{-1}]$ are all parabolic.
We conclude that $ \P_{T,  [S_1, T^{-1}]}$ is the region $\P_4$
 in Figure~\ref{fig:pleatingplanes} bounded by the line $\{ (ti, ti): t \geq 2\}$, the arc $A_1$ of $\{\Im \tau_1 \Im \tau_2 = 4\}$ from $(2i,2i)$ to $(4i,i)$ along which $T$ is parabolic; and the line $A_2 = \{ (ti, i): t \geq 4 \}$  along which 
$ [S_1, T]$ is parabolic. 
Along $A_1$,   the bending angle   $\th_{ [S_1, T^{-1}]}$  increases from zero to $\pi$ while  $\th_{T} \equiv \pi$; along $A_2$, we have $\th_{ [S_1, T^{-1}]} \equiv  \pi$ while $\th_{T}$  decreases from $\pi$ to the unattainable value  $0$.

The individual pleating ray  $\P_{\xi}$ with $\xi =  a_1T+ a_2 [S_i, T]$,  is  asymptotic to the line  $ \Im \tau_2/\Im \tau_1=  q_1(\xi)/q_2(\xi) = a_1/(a_1+a_2)$; the missing fourth 
boundary curve of 
$ \P_{T,  [S_1, T^{-1}]}$
would correspond to the non-existent  ray $ \P_{ [S_1, T^{-1}]}$ asymptotic to the line $ \Im \tau_2/\Im \tau_1 \to q_1( [S_1, T^{-1}])/q_2( [S_1, T^{-1}]) = 0$.

 A similar argument shows that $$
\overline{\P_{T, [S_2,T^{-1}]}} = \{(\tau_1,\tau_2)\in\Pi' : 1\leq \Im\tau_1 \leq \Im\tau_2,
\Im \tau_1 \Im \tau_2\geq4 \},$$ the region $\P_3$ in the figure. Thus $ \P_3$ and $\P_4$ are contained in the same 
plane in $\CC_+^2$.  Also note that the curves $$\g_{ S_1S_2^{-1}T^{-1}},  \g_{S_1^{-1}S_2T^{-1}},
\g_{[S_1, T^{-1}]}, \g_{[S_2, T]}$$ are all disjoint from $\g_{T}$ and that the ray $\P_T$ is the common boundary of all four  pleating planes $\P_1, \ldots, \P_4$.

\section{Behaviour on a pleating variety}
\label{sec:behaviour}

The heart of the proof of Theorem~\ref{thm:intromain} is the geometry of $\dd \C^+/G$ for groups
$G = G_{\xi}(\th) \in \P_{\xi}$ as $\th \to0$.   Let $\s^+= \s_i^+$ denote the geodesic representative of $\s_i$ on $\dd \C^+/G$ and 
 let $l^+_{\s}= l^+_{\s_i}$ be its hyperbolic length in  the  hyperbolic structure on  $\dd \C^+/G$.
We show that $l^+_{\s} \to 0$ as $\th \to 0$, while  $ \s^+$  becomes asymptotically orthogonal to the bending lines. From this we deduce results on the asymptotic behaviour of $\tau_1,\tau_2$. The main result of this section will be:
 \begin{theorem}
  \label{thm:intromain1}
 Fix $\xi   =  \sum_{1,2}a_i \delta_{\g_i}$ as above.  Let $G =G_{\xi}(\th)$ be the unique group in $\M$  with $\b(G) =\th  \xi$. Then  $$ \Re  \tau_i = -2p_i(\xi)/q_i(\xi) +  O(1)   \ \ {\rm and} \ \  
  \Im \tau_i = 4(1+O(c  \th) )/ \th q_i(\xi) $$
 where $c $ is a constant depending $q_1(\xi) , q_2(\xi)$ and $O(1)$ denotes a universal bound independent of $\xi$, as $\th \to 0$.
 \end{theorem}
 \begin{corollary}
  \label{cor:intromain1}
  Then  $$ |\arg  \tau_i -  \pi/2| \leq  c'\th  \ \ {\rm and} \ \   
\Bigl |\frac{q_2(\xi) }{q_1(\xi)} - \frac{ \Im \tau_1}{\Im \tau_2}\Bigr |\leq  c''\th$$
 where $c',c''>0$ are constants depending on $\xi$, as $\th \to 0$.
 \end{corollary}
Theorem~\ref{thm:intromain1} and the Corollary follow  immediately from  Propositions~\ref{thm:lengthzero2i},~\ref{thm:lengthzero2ii} and ~\ref{prop:argpi/2}. 
At the end of the section, we also prove Theorem~\ref{thm:cnvgtopml}.

 \subsection{A note on constants.}  Throughout this section, we will make many estimates of the form $X(\s_i) \leq O(\th^e)$, where $X$ is some quantity which depends on the curve $\s_i$, meaning   that  $X \leq c  \th^e$  as $\th \to 0$ for some constant $c  >0$, and $e$ is an exponent (usually $e = 1,2$ or $1/2$). 
These estimates all depend on the lamination $\xi$, so more precisely one has
 $X \leq c(\xi) \th^e$. However it is easily seen by following through the arguments that  the dependence on $\xi$ is always of the form $X(\s_i) \leq c q_i(\xi)^e\th^e$, where now $c$ is a universal  constant independent of $\xi$. 
The dependence of the constants on $\xi$ is not important for our final arguments in Section~\ref{sec:asympdirns}, but we note it as it may  be useful elsewhere.

In what follows, $X \geq  -O(\th) Y$ means that $ X \geq  - cY \th$ for some  constant $c>0$.

 \subsection{Length estimates on $\dd \C^+$.} 
 
We begin by estimating the lengths $l^+_{\s_i}$. We prove two main results, Propositions~\ref{prop:lengthzero1} and~\ref{thm:lengthzero2i},  which relate $l^+_{\s_i}$ to $\th$ and $\tau$ respectively. The first shows in particular that $l^+_{\s_i} \to 0$ as  $\th \to 0$.
 
We shall several times use the following estimate which is~\cite{smallbend} Lemma 5.4, see also ~\cite{CEpG} Theorem 4.2.10. Since the proof is so simple, we repeat it here. \begin{lemma}
\label{lem:smallbend}
 Let $\lambda$ be   a  piecewise geodesic arc  
in $\HH^3$
with endpoints $P$ and $P'$, and let $\hat \lambda$ be the $\HH^3$ geodesic
joining $P$ to $P'$.  Suppose that for all $X \in \lambda$ 
the angle  between $PX$ and $\lambda$ is  bounded in modulus by $a  \in (0, \pi/4)$. Then 
$ l_{\hat \lambda} \ge (\cos a ) l_{\lambda} $  for   all $X \in \lambda$, where $l_{\lambda} $ and $ l_{\hat \lambda}$ are the
lengths of $\lambda$ and $\hat \lambda$ respectively. 
\end{lemma}
\begin{proof}Join $P$ to a variable point $X$ on $\lambda$ between $P$ and $P'$ (see  Figure 3 in~\cite{smallbend}). If $PX$ has length $x$, the distance from $P$ to $X$ along $\lambda$ is $t$, and the acute angle between $PX$ and $\tilde S^+$ at $X$ is $\psi$, then at every non-bend point of $\lambda$, one has $ dx/dt = \cos \psi$, see~\cite{CEpG} Lemma 4.2.12.
\end{proof}

\begin{proposition} 
\label{prop:lengthzero1} Let $\xi  \in \ML_{\QQ} $ be admissible.  As usual, let $ G_{\xi}(\th)  \in \P_{\xi}$ be the unique group $G \in \M$ with $ \beta (G) = \th \xi$. Then   $$l^+_{\s_i} \leq i(\beta(G_{\xi}(\th)), \s_i)( 1 + O(\th^2)) = \th  i( \xi,\s_i) ( 1 + O(\th^2))$$ as $\th \to 0$.
 \end{proposition}
 \begin{proof}   Here and in what follows, we work in the upper half space model of $\HH^3$ and let $\dd \C^+$ denote the lift of $\dd \C^+/G$ to $\HH^3$.
 Using Lemma~\ref{lem:symmetry}  if needed, normalise such that $S_i = S: z \mapsto z+2$. Let   $\tilde S^+= \tilde S_i^+$   denote the lift  of $\s^+_i$ to  $\dd \C^+$ invariant under translation $S$.
Then
 $\tilde S^+$  is made up a number of $\HH^3$-geodesic segments which meet at bending points where $\tilde S^+$ crosses a bending line of $\dd \C^+$. 
 Choose a bending point $P=P_0$, let $P_0, P_1, \ldots, P_k, P_{k+1} =P'$ 
 be in order the bending points along one period of  $\tilde S^+$ between 
 $P_0$ and $P_{k+1} =  S(P_0)$, and let $s_i$ be the segment from $P_{i-1}$ to $P_i$, see Figure~\ref{fig:convexroof}.

Let $\phi_i$ denote the  exterior   
 angle  between the geodesic segments $s_i,s_{i+1}$ of  $\tilde S^+$ which meet at $P_i$,  measured so that $\phi_i \geq 0$  and so that $\phi_i = 0$ means that  $s_i,s_{i+1}$ are collinear. (This will be the case exactly when the bending line containing $P_i$ is contained in the interior of a flat piece of $\dd \C^+$.)
If $ \th_i$ is the bending angle of $\dd \C^+$ at  $P_i$, then  $\phi_i \leq \th_i$, see Appendix $3$.  Hence  
 \begin{equation}
 \label{eqn:anglebound}
 \sum_{i=0}^k \phi_i \leq   i(\s, \beta(G_{\xi}(\th)) = \th  i(\s, \xi). 
 \end{equation}

\begin{figure}[hbt] 
\centering 
\includegraphics[height=3.5in, viewport = 180 400  500  675 ]{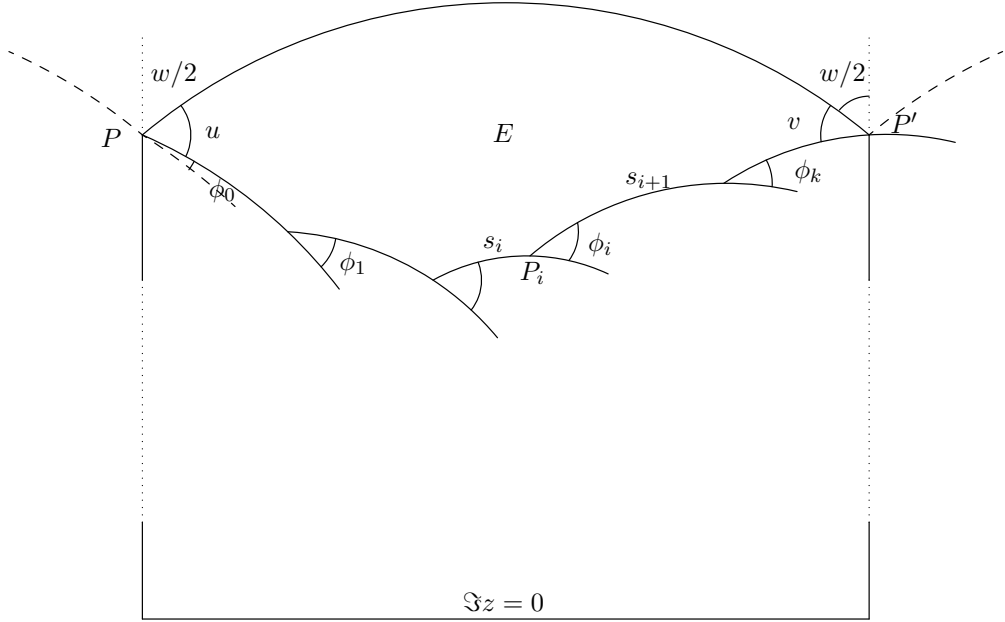} 
\caption{The geodesic representative $\tilde S^+$ of $\s$ on $\dd \C^+$.}
\label{fig:convexroof}
\end{figure} 

Now consider the triangle $\Delta$ with vertices $\infty, P,P'$.  Since $P,P'$ are at the same Euclidean height above $\CC$ in $\HH^3$, the angles in  $\Delta$  at $P$ and $P'$  are  equal, to $w/2$ say.  Let $E$ be the  (non-planar) polygon bounded by the $\HH^3$-geodesic  $PP'$ and the segments $s_i$ of $\tilde S^+$,  and let $u,v$ be the interior angles in $E$ at $P,P'$ respectively, see Figure~\ref{fig:convexroof}.
The Euclidean translation  $S$  carries the Euclidean configuration at $P$ to that at $P'=S(P)$.  It follows (using the triangle inequality in the spherical metric on the link of $P$)
that 
 \begin{equation}
 \label{eqn:anglebound1}
 \phi_0 +  u + v + w \geq \pi. 
 \end{equation}
Summing over  the interior angles of 
 $E$ gives
 \begin{equation}
 \label{eqn:gaussbonnet}
  \sum_1^k  (\pi -\phi_i ) + u + v < k\pi.
   \end{equation}
Combining  \eqref{eqn:anglebound},  \eqref{eqn:gaussbonnet} and  \eqref{eqn:anglebound1}, we find 
 \begin{equation}
 \label{eqn:anglebound2}
  \pi  -w <     \sum_0^k \phi_i  \leq  \th i(\s, \xi).
 \end{equation}

   \medskip

Now the angles in  $\Delta$  at $P$ and $P'$  are both $w/2$, so by the angle of parallelism formula, writing $d = d_{\HH^3}(P,P')$, 
$$ \sinh d/2 = \cot w/2,$$  so that  from~\eqref{eqn:anglebound2} we find
$$ \sinh d/2 \leq  \th i(\s, \xi)(1 + O( \th^2))/2.$$

Finally, we claim that $$ l^+_{\s} \leq  d(1 + O(\th^2))$$  for small $\th$, from which the result is immediate. This follows easily from Lemma~\ref{lem:smallbend}. We have only to see that the angle $\psi$ between the line $PX$  from $P$ to any point $X$ on $\tilde S^+$
is bounded above by $\th i(\s,\xi)$.  This follows since  along the interior of any segment $\psi$ decreases as $x$ increases, and since at the bend point $P_i$ it increases by at most $\theta_i$. This gives $d \geq l_{\s}^+(1-O(\th^2))$ and the result follows.
   \end{proof}

 \begin{corollary}
\label{cor:noalglim} Let $ \xi  \in \MLQ$ be admissible and suppose that  $ G(\t_1,\t_2)$  is the unique  group $G_{\xi}(\th)$  in $  \M$ with $\b(G )=   \th \xi$.
 Then $1/\Im \tau_i \leq  O(\th), i=1,2$,  as $ \th  \to 0$. Moreover the groups $G_{\xi}(\th) $ have no algebraic limit as 
$ \th  \to 0$. \end{corollary}
 \begin{proof} 
As above, we work in the upper half space model and we assume $G = G_{\xi}(\th)$ normalised so that $S=S_i$ is the translation $z \mapsto z+2$. Define $P,P'$ as before. Let $k$ be the Euclidean height $k$ of $P$ above $\CC$. Then $1/k = \sinh d_{\HH^2}(P,P')/2$ while on the other hand,
$  d_{\HH^2}(P,P') \leq l_{\s^+} \leq O(\th)$. It follows that 
$1/k \leq O(\th)$.

Now $P$ lies on a bending line $\zeta$ of $\dd \C^+$ which is contained in a 
support plane to $\dd \C^+$. This plane is a hemisphere $H \subset \HH^3$ which meets $\Chat$ in a circle $C$. Since the convex core is contained entirely to one side of $H$, there are no limit points in one of the two discs in $\Chat$ bounded by $C$.  Since the horizontal lines $\Im z = 0$ and $\Im z = \Im \tau_i$  are contained  in  the limit set $\Lambda$, and since the half planes $\Im z < 0$ and $\Im z > \Im \tau_i$ are contained in $\Omega^-$, it follows that $C$ is contained in $\{ z: 0 \leq \Im z \leq   \Im \tau_i\}$
and hence that its diameter is at most $\Im \tau_i$.  We deduce $2/ \Im \tau_i \leq 1/k$ so that $1/ \Im \tau_i \leq O(\th)$ as claimed.

To prove that the algebraic limit of a sequence of groups does not exist, it is enough to show that the trace of some element becomes infinite. The result follows on recalling from Section~\ref{sec:concrete} that $\tr {[T,S_i^{-1}]} = \tau_j^2 + 2, j \neq i \mod 2$.
\end{proof}

Now we establish a more precise link between $l^+_{\s_i} $ and $\Im \tau_i$:
 \begin{proposition}
\label{thm:lengthzero2i}  Let $ \xi  \in \MLQ$ be admissible and suppose that  $ G(\t_1,\t_2)$  is the unique  group $G_{\xi}(\th) \in   \M$ with $\b(G )=   \th \xi$.
 Then
  $$\Im \tau_i (1- O(\th))  \leq  4/l^+_{\s_i} \leq \Im \tau_i (1+ O(\th))$$ as $\th \to 0$. 
   \end{proposition}

\begin{remark}
\label{remark:kra}  {\rm 
Kra~\cite{kra} p. 568  gives the estimate  $ \Im \tau_i -1 \leq 2\pi/l_{\s_i}(\Omega^+) \leq   \Im \tau_i$
for the hyperbolic length $l_{\s_i}(\Omega^+)$ of the geodesic representative of $\s$  on $\Omega^+/G$. 
Combining this with Sullivan's theorem~\cite{EpM}, gives an alternative proof of the final statement of Corollary~\ref{cor:noalglim}.
The discrepancy of the $\pi/2$ factor with the estimate in Proposition~\ref{thm:lengthzero2i}
 is perhaps somewhat surprising, as one would expect that, since $G$ becomes asymptotically Fuchsian, the structures on $\dd \C^+/G$  and $\Omega^+/G$ would be asymptotically equal. For a quick confirmation  of Kra's estimate, note  that $\Omega^+$  contains an $S$ invariant  strip of width $\Im \tau -1$,  so that there is an annular collar $A$  of approximate modulus $\Im \tau/2$ around  $S$ on $\Omega^+$.  Kra's estimate follows from the formula $\mod A = \pi/l_{\s}(\Omega^+) + O(1)$, see for example~\cite{minskyproduct} p.255.
}
 \end{remark}

 We begin the proof of Proposition~\ref{thm:lengthzero2i} with two lemmas.
 
  \begin{lemma}
\label{lem:eta} Let $\xi \in \MLQ$ be admissible. 
Normalize $G_{\xi}(\th)$ so that $S_i$ is translation $z \mapsto z+2$. 
 Let $\g$ be a bending line of the pleating lamination $\xi$ which intersects $\s_i$. Then there is a lift $\tilde \gamma$  of $\g$   with  endpoints $\g^{\pm}$  such that 
  $|\Re(
 \gamma^+ - \g^-) | \leq 2$ and 
  $\Im \tau_i-1 < |\Im(
 \gamma^+ - \g^-) |   < \Im \tau_i$.    \end{lemma} 
 \begin{proof} We continue with the set up of Corollary~\ref{cor:noalglim}.  We can choose a lift $\tilde \gamma$ of $\gamma$ to be  identified with the bending line $\zeta$ in that proof. The
endpoints $\gamma^{\pm}$ of $\tilde \gamma$ lie in $C \cap \Lambda$, where 
  $C$ is a circle contained in the strip $B$ between the lines $\Im z = 0$ and $\Im z = \Im \tau_i$ and as usual $\Lambda$ is the limit set of $G$.  By Proposition~\ref{thm:inM}, $\Lambda$ is contained in the union of the two strips
$ 0 \leq \Im z \leq 1/2$ and $ \Im \tau_i -1/2 \leq \Im z \leq \Im \tau_i$.  (This proposition holds on the hypothesis that $ \Im \tau_i > 1$, which  holds for small $\th$ by Corollary~\ref{cor:noalglim}.) Moreover  there are no points of $\Lambda$  in the intersection of the two open disks $D$ and $S(D)$ bounded by $C$ and $S(C)$
and inside $B$.  
We deduce that $\zeta$  either has Euclidean height at most $O(1)$, or has  endpoints one of which is in the rectangle $|\Re z| < 1,  0 \leq \Im z \leq 1/2$ and the other in the rectangle $|\Re z| < 1, \Im \tau_i -1/2 \leq \Im z \leq \Im \tau_i$.  Since $P \in \tilde \gamma$ and since the Euclidean height of $P$ is large  by Corollary~\ref{cor:noalglim}, the result follows.
\end{proof}

Continuing with the assumption that  $\xi \in \MLQ$ is admissible and 
that $G_{\xi}(\th)$ is normalized  so that $S_i$ is translation $z \mapsto z+2$,  
we call any lift of a bending line satisfying the conditions of Lemma~\ref{lem:eta},    \emph{good}. 
 Combining  with Corollary~\ref{cor:noalglim} we obtain easily:
  \begin{corollary}
\label{lem:eta1} 
Let $\tilde \g$ be a good lift of a bending line $\g$ and set 
 $ \gamma^+ - \g^- =2r e^{i\a} $, where without loss of generality we take  $\Im( \gamma^+ - \g^-) >0$. Then 
 $r = (1+O(\th))\Im \tau_i/2$ and $|\pi/2 -  \a | = O(\th)$ as $\th \to 0$.
\end{corollary} 
 
 \begin{lemma}
 \label{lem:cxdist} Let $\tilde \g$ be a good lift of a bending line of the pleating lamination which intersects $\tilde S^+_i$. 
Then  the complex distance $D$ 
 between $\tilde \g$ and  $S_i(\tilde \g)$ is given by 
  $$ -\frac{1} {r^2e^{2i\a}} = \sinh^2 D/2.$$
 \end{lemma}
 \begin{proof} Recall that the complex distance between two axes is $d + i \psi$ where $d$ is the real perpendicular distance and $\psi$ is the rotation of one axis relative to the other along their common perpendicular, see e.g.~\cite{swolpert} for details.
  We use the cross ratio formula for complex distance. Let $z_1,z_2$ and $w_1,w_2$ be endpoints of two oriented geodesics in $\HH^3$ at complex distance $D$. We can conjugate  so that  $z_1,z_2$ move to $1,-1$ and  $w_1,w_2$ move to $e^D,-e^D$ respectively.  Then
 $$ [z_1, w_1,w_2,z_2] := \frac{z_1-w_2}{z_1-w_1} \cdot \frac{w_1-z_2}{w_2-z_2} = \coth^2D/ 2.$$

Applying this formula to the  
  four endpoints of  $\tilde \g$, $S_i(\tilde \g)$ which are at points $\g^{+}, \g^-$ and $\g^{+}+2, \g^-+2$  respectively gives    $$ \coth^2 D/2 = [\g^{+}, \g^{+}+2, \g^-, \g^-+2]$$  which simplifies to   the claimed result.
  \end{proof}

 \noindent{\sc {Proof of Proposition}~\ref{thm:lengthzero2i}.} Pick a good lift $\tilde \gamma$ of a  bending line, and let  $P$  be the point where $\tilde \gamma$ meets  $\tilde S^+$. Let $Q$ be the highest point of $\tilde \gamma$ and let $Q' = S(Q)$, so  that the Euclidean height of both $Q,Q'$ above $\CC$ is $r$.
Writing $d_{\dd \C}$ for the induced hyperbolic metric on $\dd \C^+$, we have $$l_{\s^+}  \leq d_{\dd \C}(Q,Q')  \leq d_{\HH^3}(Q,Q')(1+O(\th^2))$$
where the first inequality follows since the curve on $ \dd \C^+$ from $Q$ to $Q'$ is in the homotopy class of $\tilde S^+$, and the final one follows by Lemma~\ref{lem:smallbend}.

 Now $$\sinh d_{\HH^3}(Q,Q')/2 = 1/r$$  and $\Im \tau -1 \leq 2r \leq \Im \tau$ by Lemma~\ref{lem:eta}.
 Thus $$d_{\HH^3}(Q,Q')/2 \leq 1/r \leq 2/(\Im \tau -1)$$
from which we deduce
$$l_{\s^+} \leq 4 \frac{(1+O(\th^2))}{\Im \tau -1} = \frac{4}{\Im \tau} \Bigl( 1+O(1/\Im \tau)  \Bigr) \Bigl( 1+O(\th^2)  \Bigr) =  \frac{4}{\Im \tau} \Bigl( 1+O(\th)  \Bigr) $$ by 
Corollary~\ref{cor:noalglim}.
Hence \begin{equation}
\label{eq:lower}
 \Im \tau/4 \leq 1/l_{\s^+} (1+ O(\th)).
\end{equation}  

To find an upper bound for $1/l_{\s^+}$,  we use Lemma~\ref{lem:cxdist}. 
Writing $D = d + i \psi$, note that since $d \leq l_{\s^+}$ it is enough to find an upper bound on $1/d$.
Comparing real and imaginary parts in the   formula of  Lemma~\ref{lem:cxdist} gives
$$ \sinh  d/2 \cos \psi/2 = \pm \sin \a/r \  {\rm and} \  \cosh  d/2 \sin \psi/2 = \pm \cos \a/r.$$
 
Since $d \leq l_{\s^+} \leq O(\th)$ we find
$$\frac{2}{d} \leq  \frac{ ( 1+O(\th^2)   )}{\sinh  d/2}  = \frac{r( 1+O(\th^2) )}{|\sin \a|}|\cos \psi/2|.$$

By Corollary~\ref{lem:eta1}, we have 
$1/|\sin\a|  = 1+ O(\th)$.
Thus
$$     2/d \leq r(1+O(\th)) $$
from which
\begin{equation}
\label{eq:upper}
2/l_{\s^+} \leq \Im \tau (1+O(\th))/2.
\end{equation}

Inequalities \eqref{eq:lower} and \eqref{eq:upper} together complete the proof.
\qed

 \subsection{Asymptotic orthogonality} 
\label{sec:asymporthog}

Propositions~\ref{prop:lengthzero1}  and~\ref{thm:lengthzero2i} 
are not enough to give   the detailed asymptotics of Theorem~\ref{thm:intromain1}. We also need the following more
 refined comparison: 
 \begin{proposition}
\label{thm:lengthzero2ii} 
  Along the pleating variety $\cal P_{\xi}$, we have  
$$ \th i(\xi,\s) (1- O(\th)) \leq  l^+_{\sigma} \leq   \th i(\xi,\s) (1+O(\th^2))$$    as $\th \to 0$.
 \end{proposition}

This result  is a direct consequence of the fact that asymptotically, $\tilde S^+$ becomes orthogonal to the bending lines, see Proposition~\ref{prop:orthog}. 
 The intuition for this is the following. Suppose
that $\tilde S^+$ were actually
perpendicular to all bending lines. Then each good lift of a bending line cut by $\tilde S^+$ would have Euclidean height between $ \Im \t/2 -1$ and $\Im \t/2$,  and in the proof of Proposition~\ref{prop:lengthzero1}, Equations~\eqref{eqn:anglebound}  and~\eqref{eqn:anglebound1} would be equalities.  Since $E$ has area $O(\th^2)$, in this situation \eqref{eqn:anglebound2} becomes an equality up to  $O(\th^2)$.
This situation actually pertains in the case in which there is  unique bending line in the class of $T$ (in which case $E = \emptyset$), see  Example 1 in Section~\ref{sec:examples}.

  \begin{lemma} 
  \label{lem:commonperp1}
Let $\tilde \g$ be a good lift of a bending line and as above, let $P,P'$ be the points at which   $\tilde S^+$meets $\tilde \g$ and $S(\tilde\g)$ respectively. If  $K$ is the Euclidean centre on $\CC$ of the semicircle $\tilde \g$, then $ \angle PKQ  =  O(\sqrt \th)$.
 \end{lemma}
  \begin{proof} 
  As in Corollary~\ref{cor:noalglim}, let $k$ denote the Euclidean height of $P$. We have $\sinh d_{\HH^3}(P,P')/2 = 1/k$ and $1/k < O(\th)$. Hence
 $ d_{\HH^3}(P,P')= 2(1+O(\th^2))/k$.  We will estimate $k/r =  \cos   \angle PKQ$.

By Lemma~\ref{lem:smallbend}:
 $$d_{\HH^3}(P,P') \leq l_{\s^+}  \leq (1+O(\th^2)) d_{\HH^3}(P,P')$$
 and hence 
\begin{equation*}
\label{eq:1}
 l_{\s^+} = 2(1+O(\th^2))/k.
\end{equation*}
  
On the other hand, by Proposition~\ref{thm:lengthzero2i} and Lemma~\ref{lem:eta}, 
  $$  2/l_{\s^+} = r \bigl( 1+O(\th) \bigr).$$
  Combining these two equations gives
  $$| k/r - 1| = O(\th)$$
  so that  
 $   \angle PKQ \leq O(\sqrt \th)$  as claimed. \end{proof}

\medskip 
Now we prove our result on 
 asymptotic  orthogonality.

 \begin{proposition}
\label{prop:orthog} 
Along the pleating variety $\cal P_{\xi}$,  the curve $\tilde S^+$ is asymptotically orthogonal to the bending lines as $\th \to 0$. More precisely, suppose that $\tilde  S^+$ meets an (oriented) bending line $\tilde \gamma$ at a point $P$ so that the acute angle  between  $ \tilde S^+$ and $\tilde \gamma$ is $\psi(P)$. Then $|\psi(P)-\pi/2| \leq O( \sqrt{\th})$ as $\th \to 0$.
 \end{proposition}
 \begin{proof} As usual, we may suppose that  $\tilde \gamma$ is a good lift of a bending line, and we let $P' = S(P)$.
  Let $\kappa$ denote the geodesic arc in $\HH^3$ from  $P$ to $ P'$.
  Let $\bf u, \bf v,\bf w$ be forward pointing unit vectors at $P$  along $\tilde S^+$, $\kappa$, 
 and  $\tilde \g$ respectively, where $\kappa, \tilde S^+$ are given the same orientation and the orientation of  $\tilde \g$ is from $\g^-$ to $\g^+$, where $\Im  \g^- < \Im \g^+$.
Let   $\Psi(\bf x, \bf y) $ denote the angle  between the vectors $\bf x$  and $\bf y$ at $P$, so that $ \psi(P) = \Psi(\bf u, \bf w)$.

From  Equation~\eqref{eqn:gaussbonnet} in  the proof of   Proposition~\ref{prop:lengthzero1}, we have $\Psi({\bf u}, {\bf v})  \leq \th i(\s,\xi) $. 
Thus it will suffice to show that  $|\Psi({\bf v},{ \bf w})  - \pi/2 | = O(\sqrt \th)$.
We prove this by a finding a sequence of small rotations which, when applied to 
$\bf v$ and $ \bf w$, results in    a pair of perpendicular vectors. 

 Let $\Gamma$ be the vertical plane  containing $\tilde \gamma$ and let $Z$ be  the footpoint of the perpendicular from $P$ to $\CC$. As in Lemma~\ref{lem:commonperp1}, let  $K$ be  the footpoint of the perpendicular from the highest point $Q$ on $\tilde \g$ to $\CC$. Let  $\Theta_1$ be  rotation by an angle
$ \angle PKQ$  about a line perpendicular to $\Gamma$ through $P$, 
 so that
 $\Theta_1({\bf w})$ is horizontal (i.e. parallel to the base plane $\CC$).  By Lemma~\ref{lem:commonperp1}, $\angle PKQ = O(\sqrt \th)$, so that $\Theta_1({\bf w}) = {\bf w} + O(\sqrt  \th)$.

 The vectors $\bf w$ and $\Theta_1({\bf w}) $ are  in the 
plane $\Gamma$ containing $\g$, which is at angle $\a$ to the real axis  and hence to the vertical plane $\Pi$ containing $P$ and $P'$.  These two planes intersect in the line $ZP$. Thus if $\Theta_2$ denotes  anticlockwise rotation by $\pi/2 - \a$ about $ZP$, then   $\Theta_2(\Theta_1(\bf w))$ is orthogonal to $\Pi$ at $P$.
Since by Lemma~\ref{lem:eta}, $\pi/2 - \a = O(\th)$, we have
${\bf w} = \Theta_2(\Theta_1(\textbf {w})) +  O(\sqrt \th)$.

Finally, let $\zeta$ be the angle between $PZ$ and the arc $\kappa$ at $P$, so that 
$\sinh d_{\HH^3}(P,P') /2  = \cot \zeta$. Since $d_{\HH^3}(P,P') \leq l_{\s^+}$
this gives $ \pi/2 - \zeta = O(\th)$. 
If $\Theta_3$ is  rotation by an angle
$\pi/2- \zeta$  about a line perpendicular to $\Pi$ through $P$, then 
 $\Theta_3( {\bf v})$ is horizontal and lies in $\Pi$ and  $\Theta_3( {\bf v}) = {\bf v} + O(\th)$.
 
By construction  $ \Theta_2(\Theta_1(\textbf {w}))$ is orthogonal to $\Theta_3( {\bf v})$ , so that putting these three estimates together we find 
 $|\Psi({\bf v}, {\bf w})  - \pi/2 | = O(\sqrt \th)$ as claimed.
 \end{proof}

Finally, we can prove Proposition~\ref{thm:lengthzero2ii}.

\medskip

\noindent{\sc {Proof of Proposition}~\ref{thm:lengthzero2ii}.}
 In view of Proposition~\ref{prop:lengthzero1}, it will be enough to find a bound
 $l_{\s^+} \geq \th i(\xi,\s) ( 1- O(\th))$ as $\th \to 0$.

 We work in the vertical plane $\Pi$ containing $P$ and $P'$. This plane intersects $\dd \C^+$ in a path which is a union of $\HH^3$-geodesic segments similar to but not the same as the geodesic path $s_0, s_1, \ldots, s_k$ of Proposition~\ref{prop:lengthzero1}.
 Denote the bending points along this path $P = \hat P_0, \hat P_1, \ldots, \hat P_{k+1}$ where $\hat P_{k+1} = S(P_0) = P'$, and let $\hat s_i$ be the segment from $\hat P_{i-1}$ to $\hat P_{i}$.

Let $\hat \phi_i$ denote the  exterior   
 angle  between the segments $\hat s_i, \hat s_{i+1}$ which meet at $\hat P_i$,  measured in the same way as $\phi$ in the proof  of Proposition~\ref{prop:lengthzero1}, and let $\th_i$ be the bending angle between the support planes of $\dd \C^+$ which meet at $\hat P_i$.
We will prove below that 
  \begin{equation}
 \label{eqn:epangles} 
 \hat \phi_i/\th_i  = 1+O(\th).
  \end{equation} 
Denote by $l(\hat s_i)$ the hyperbolic length of $\hat s_i$.
We observe:
$$  l^+_{\s} \geq d_{\HH^3}(P,P') \geq (1-O(\th^2))\sum l(\hat s_i)$$
where the last inequality follows by Lemma~\ref{lem:smallbend}  since  from~\eqref{eqn:epangles} 
$\hat \phi_i = O(\th)$ for all $i$. 
To estimate $l(\hat s_i)$, join each point $\hat P_i$ to $\infty$ in the plane $\Pi$ and let $y_i = \angle \hat P_{i-1}\hat P_i\infty$ and $x_i = \angle \infty \hat P_{i}\hat P_{i+1}  $, so that (since all angles are measured in the plane $\Pi$), $x_i + y_{i} + \hat \phi_i = \pi$.
  The formula for the length of the finite side of  triangle with angles $x,y, 0$  (see~\cite{beardon} theorem 7.10.1)
gives
$$\sinh l(\hat s_i) = \frac{\cos x_i + \cos y_{i+1}}{\sin x_i + \sin y_{i+1}}.$$
We also prove below that 
\begin{equation}
\label{eqn:almostbisect}
|\pi/2 -x_i | = O(\th) \ {\rm and} \ |\pi/2 -y_i | = O(\th).
\end{equation}
 This gives
$$  l(\hat s_i)  = (\pi - x_i -y_{i+1})(1 + O(\th^2))$$
from which 
$$ \sum_0^k l(\hat s_i) = (\sum_0^k \hat \phi_i ) (1 + O(\th^2)) = (\sum_0^k   \th_i ) (1 + O(\th))$$ by~\eqref{eqn:epangles}.
Hence 
 $$ l^+_{\s}  \geq   \th i(\xi, \s) (1-O(\th ))$$
 as $\th \to 0$. This completes the proof, modulo the proofs of ~\eqref{eqn:epangles} and
 ~\eqref{eqn:almostbisect}.

 \smallskip
 
 \noindent \textbf{Proof of~\eqref{eqn:epangles}}.
  This follows from Appendix A.4 in~\cite{EpM}. Suppose that two planes $J_1,J_2$ meet at an angle $\zeta$ along a line $L$, so that $\zeta$ is the angle between the lines of intersection of $J_1,J_2$ with  the plane $H$ orthogonal to $L$.  Suppose that $H'$ is another plane slightly skewed to $H$, and let 
  $\zeta^*$ be the angle between the lines of intersection of $J_1,J_2$ with  $H'$.
Setting things up so that $H$ has unit normal $(0,0,1)$ and so that the bisector of the planes orthogonal to $L$ is the vector $(1,0,0)$, the result of~\cite{EpM} gives an estimate of $\zeta^*/\zeta$ in terms of the unit normal $(x_1,x_2,x_3)$
  to $H'$ for small  $x_1,x_2$. In fact if $x_i = O(\epsilon), i=1,2$, then we find easily  either from the formula for 
  $\tan \zeta^*/2 / \tan \zeta/2$ in terms of the $x_i$ on p. 246, or from Theorem A.4.2 on p. 247,  that  $\zeta^*/\zeta  = 1+O(\epsilon^2)$.  
  
To apply the theorem in our case, we want to find the angle between the support planes which meet along the bending line at $P_i$, measured in the plane $\Pi$ which is slightly skewed to the plane orthogonal to the lift $\tilde \gamma_i$ of $\gamma$ through $\hat P_i$.  
In the above set up, the unit vector along $\tilde \gamma$ at $\hat P_i$ is ${\bf e_3} = (0,0,1)$, and the line bisecting the planes along $\gamma$ is ${\bf e_1} = (1,0,0)$.

Let $ \bf e_3', \bf e_1'$ be  unit vectors at $\hat P_i$ orthogonal to  $\Pi$, and in $\Pi$ pointing vertically upwards, respectively. It will be sufficient to show that 
${ \bf e_3'} ={  \bf e_3 }+ O(\sqrt \th)$ and  $ { \bf e_1' }= {\bf e_1} + O(\sqrt \th)$. 
Now if $\hat P_i$ were replaced by the similar configuration at the point $P_i$, then the first 
result would follow from Corollary~\ref{lem:eta1} and Lemma~\ref{lem:commonperp1}, while the second would follow
from the proof of Proposition~\ref{prop:lengthzero1}, since as illustrated in Figure~\ref{fig:convexroof}, $ \bf e_1' $ bisects the angle between $s_i$ and $s_{i+1}$ up to $O(\th)$.  

In the move  from $P_i$ to $\hat P_i$, the  estimates for  $ \bf e_3', \bf e_1'$ will change  by terms on the order of $ \angle P_i K_i \hat P_i$,  the angular distance from $P_i$ to $\hat P_i$
along $\tilde \gamma_i$. (Here $K_i$ is the centre of the Euclidean semi-circle $\tilde \gamma_i$.) We estimate this as follows.
Let $d_i$ be the perpendicular distance between the plane $\Pi$ through $P$ parallel to the real axis, and the parallel plane $\Pi_i$ through $P_i$. Since $P$ is joined to $P_i$ by segments $s_0, s_1, \ldots s_{i-1}$ along $\tilde S^+$, and since it follows  from ~\eqref{eqn:gaussbonnet} in Proposition~\ref{prop:lengthzero1} that each segment $s_j$ makes an angle at most $O(\th)$ with the plane $P_{j-1}$, we find
$$ d_i \leq \sum_0^{i-1} l(s_j) O(\th) = O(\th^2).$$

Let $Z_i, \hat Z_i$ denote the footpoints  in $\CC$ of  the vertical lines through $P_i, \hat P_i$. The estimate on $d_i$ combined with  Corollary~\ref{lem:eta1}  gives $|Z_i -  \hat Z_i | = O(\th^2)$. 
It follows  that the error in replacing  $P_i$ by $\hat P_i$ is of a lower order than those already obtained, and  we conclude that ${ \bf e_3'} ={ \bf e_3} + O(\sqrt \th)$ and  $ {\bf e_1'} ={ \bf e_1} + O(\sqrt \th)$ as claimed. 
\medskip

\noindent \textbf{Proof of~\eqref{eqn:almostbisect}}.
Let $L$ be a line from  $\infty$ to a variable point $X$ on some segment $\hat  s_i$, and let $x = x(X)>0$ be the acute angle  in the complementary pair $ \angle \hat P_{i-1} X\infty$ and $ \angle \infty X\hat P_{i}$.  We want to show that  $|\pi/2 -x | = O(\th)$.
 Since  $P_0$ and $S(P_0)$ are at the same Euclidean height, there is certainly some $\hat  s_{i_0} $ containing a point $X_0$ at which the tangent to $\hat  s_{i_0} $ is horizontal, so $x(X_0)=\pi/2$. Moving away from this point in either direction, $x(X)$ decreases according to the formula $ \sinh t = \cot x( X(t))$, where  the point
$X(t)$ is at distance $t$ from $X_0$ along $\hat  s_{i_0}$. The change in angle at a bend point is $\hat \phi_i = \th_i (1+O(\th))$ by~\eqref{eqn:epangles}, and on the adjacent segment the argument proceeds as before. Since (using Lemma~\ref{lem:smallbend} again) $\sum_i l(\hat  s_i) \leq O(\th)$, the result follows.
\qed

 \subsection{Twisting}
 \label{sec:twist}
To complete the proof of Theorem~\ref{thm:intromain1},   it remains to bound $\Re \tau_i$. We have:
 \begin{proposition}
\label{prop:argpi/2} Let $\xi \in \MLQ$ be admissible and suppose that $\g \in \S$ is  contained in the support of $\xi$. Then if 
 $(\tau_1,\tau_2) \in \cal P_{\xi}$,  we have  $\Re \tau_i = -2p_i(\g) / q_i(\g) + O(1)$
and hence   in particular $|\arg  \tau_i - \pi/2| = O({\th })$ as $\th \to 0$.
 \end{proposition}
Since  this result holds for \emph{any}   $\g$ contained in the support of $\xi$ we have:
  \begin{corollary}
\label{cor:argpi/2} Suppose that $\g,\g' \in \S$ are  supported on a common admissible lamination $\xi$. Then 
$|p_i(\g) / q_i(\g) -p_i(\g') / q_i(\g')|  \leq 10$ and hence $\Re \tau_i = -2p_i(\xi)/q_i(\xi) + O(1)$.
 \end{corollary}
The condition that $\g,\g'$ are  supported on a common admissible lamination is equivalent to the condition that together with $\s_1,\s_2$ they fill up $\Sigma$.  The value $10$ is obtained by following  through the constants in the argument below; it could  certainly  be improved by more careful  inspection. 
 
We prove this result using the concept of the  twist of one geodesic around another  following Minsky \cite{minskyproduct}.   Suppose given a hyperbolic metric $h$ on the  surface $\Sigma$. The \emph{twist} $\tw_{\b}(\g,h)$ of a curve $\g$ about another curve  $\b$ is defined as follows. Let $p$ be an  intersection point of $\g$ with $\b$.
Let $P$ be a lift to  $\HH^2$ of $p$ and let $\tilde \g, \tilde \b$ be the lifts of  $\g,\b$ through $P$. Orient $\tilde \g, \tilde \b$ with positive endpoints $Z,W$ respectively  on $\dd \HH^2$ so that the anticlockwise arc from $Z$ to $W$ does not contain the other two endpoints.  Let $R$ be the footpoint  of the  perpendicular from $Z$ to $\tilde \b$. Let $t$ be the oriented distance $PR$, where $t>0$ if  $R$ follows $P$ in the positive direction along $\tilde \b$ and  $t \leq 0$ otherwise.  One verifies, see~\cite{minskyproduct} Lemma 3.1, that $ t/l_{\b}(h)$ is independent up to an additive error of $1$ of the choices made, including the choice of $p$.
Finally, define $\tw_{\b}(\g,h) = \inf {\it t/l_{\b}(h)}$, where we take the infimum over all possible choices of lifts as above.

Note that the twist is independent of the orientation of $\b,\g$ but depends on the choice of hyperbolic metric $h \in  \teich$, where $  \teich$
is  the Teichm\"uller space of $\Sigma$. However:
\begin{lemma}
\label{lemma:reltwist}  (\cite{minskyproduct} Lemma 3.5, see also~\cite{crs} Sec. 4.3.) 
 For any two  $\g_1, \g_2 \in \S$,  the \emph{relative twist}
$\tw_{\a}(\g_1,h)- \tw_{\a}(\g_2,h)$ is independent of $h \in  \teich$, up to a bounded additive error of $1$.  
\end{lemma}
We define the \emph{signed relative twist} of $\g_1, \g_2$ with respective to $\b$ to be 
$i_{\b}(\g_1, \g_2) =\inf_{h \in \teich}   \tw_{\b}(\g_1,h)- \tw_{\b}(\g_2,h)$.
Here is a useful way of computing it:
\begin{lemma}
\label{lemma:reltwist1}  
Let  $\g_1, \g_2 \in \S$ and let $\tilde \g_1, \tilde \g_2$ be  lifts of $\g_1,\g_2$ which cut  the fixed axis
$\tilde \b$  corresponding to  $\b$, and let $b \in \Gamma$ be the primitive element whose axis is $\tilde \b$ and whose attracting fixed point is the positive endpoint of $B$, where $ \Gamma$ is the Fuchsian group uniformising $h$. Then 
 $\tw_{\b}(\g_1,h)- \tw_{\b}(\g_2,h)$ is equal in magnitude to the number of times the images 
 $b^n(\tilde \g_1), n \in \ZZ$ intersect $\tilde  \g_2$, up to a bounded additive error of $1$. 
The sign is negative if $b(\tilde  \g_1)$ follows $\tilde  \g_1$ in the positive direction along $\tilde  \g_2$ and positive otherwise. 
\end{lemma}
\begin{proof} This is clear from the definition in a metric  in which 
$\tilde \g_1$ is orthogonal to $B$. Whether or not two axes intersect, depends only on the relative position of their endpoints round the boundary at infinity $\dd \HH = S^1$.  Since a quasiconformal deformation  of $\HH$ induces a homeomorphism of   $\dd \HH$, we deduce that the relative positions of endpoints of axes are independent of the metric $h$, from which the result follows. \end{proof}

We shall prove Proposition~\ref{prop:argpi/2} by computing  $ i_{\s_i}(\g,{T})$ in two different ways, where  as usual $ i_{\s_i}(\g,T)$ means $ i_{\s_i}(\g,\g_{T})$ where $\g_{T} \in \S$ is the curve corresponding to the generator $T \in \pi_1(\Sigma)$. We have:  
 \begin{lemma}
\label{lem:pqtwist} Suppose that $\g \in \S$ has canonical coordinates $${\bf i}(\g)= (q_1(\g), p_1(\g), q_2(\g), p_2(\g)).$$ Then 
 $ i_{\s_i}(\g,T)= -p_i(\g) / q_i(\g) + O(1)$. 
\end{lemma}
 \begin{proof} 
 We work in the fundamental domain $\Delta$ of $\Sigma_{1,2}$ and label the sides as
  in Figure~\ref{fig:funddomain}.  
  We also suppose that $p_1 \geq 0$.
Let $\tilde A$ be a horizontal  strip   joining  $s_{S_1}$ to $s_{S^{-1}_1}$,  shown shaded in the figure. This projects to an annular neighbourhood $A$ of $\s_1$ on $\Sigma$. We may take our lift $\tilde \s_1$ of $\s_1$ to be the centre line of $\tilde A$ and the lift $\tilde \g_T$ of the curve $\g_T$ to be the arc joining  the midpoints of  $s_{T}$ and $s_{T^{-1}}$. This   intersects $\tilde A$ in a single  arc $\lambda$ which  joins the boundaries  $\dd^-\tilde A, \dd^+ \tilde A$ of $\tilde A$. 
By the previous lemma, to compute $i_{\s_1}(\g,T)$ we have to examine how many images $S_1^n(\tilde \g_T)$ cut a fixed lift $\tilde \g$ of $\g$, equivalently how many images $S_1^n(\tilde \g)$ cut $\tilde \g_T$.

The lift  $\tilde \g$ appears as a collection  of disjoint arcs joining sides of $\Delta$, the numbers of arcs joining particular pairs of sides being determined by the canonical coordinates $\i(\g)$. It is not hard to see that 
the magnitude of the relative intersection number $i_{\s_i}(\g,T) $ is, up to an additive error of $1$, the number of times that a connected component $\kappa$ of $  \g \cap A$  cuts the projection $\pi( \lambda)$  of $\lambda$ to $\Sigma$. For convenience, replace $\lambda$ by $L = s_{S_1} \cap \tilde A$. This changes the intersection number by at most $2$.

Denote the lift of $\kappa$ to $\Delta $ by $\tilde \kappa$. 
 Clearly $ \tilde \kappa$ contains at most two `corner arcs' (see Section~\ref{sec:simplecurves}) of $\tilde \g \cap \Delta$,  which  can be ignored in our count. 
 If $p_1 \leq q_1$ then $\tilde \g$ contains no horizontal strands and  $i(\kappa, L) \leq 1$.

Now suppose   $p_1 > q_1$. In this case $\tilde \g\cap \Delta$ contains
$m >0$  horizontal arcs running across $\tilde A$ joining  $L$ to $S_1(L)$.
Thus after entering $ A$ across $\dd^-  A$,  the component $  \k$  travels around
$A$ cutting   $\pi(L)$  either $m$ or $m+1$ times  before exiting across  $\dd^+  A$.  
(This is the well known combinatorics of simple curves crossing a cylinder.)
 The total number of such connected components is $i(\s_1,\g) = q_1$.
On the other hand, the total number of strands of $\tilde \g \cap \Delta$ which meet $L$ is by definition $p_1$.
Thus $mq_1 \leq p_1 < (m+1)q_1 + q_1$ and so $m =  [p_1/q_1] + O(1)$. 
Finally, we check from the definition that  in the obvious metric $h_0$ on $\Delta$ in which $\tilde \s_1$ is orthogonal to $\tilde \g_T$, we have $i_{\s_1}(\g,h) < 0$ while $i_{\s_1}(\g_T,h)  = 0$. Thus  $ i_{\s_i}(\g,T)= -p_i(\g) / q_i(\g) + O(1)$
as claimed. 
 
The arguments for $p_1<0$ and for $S_2$ are similar. \end{proof}

  \noindent  {\sc Proof of Proposition~\ref{prop:argpi/2}.}  
From Lemma~\ref{lem:pqtwist} we have $i_{\s_i}(\g,T) = - p_i(\g) / q_i(\g) + O(1)$.
On the other hand we can also compute $i_{\s_i}(\g,T)$ as follows. As usual, after normalizing suitably  let $\tilde S^+ = \tilde S_i^+$ be the lift of $\s_i$ to $\dd \C^+$ which is invariant under $S_i:z \mapsto z+2$. If $\tilde \g$ is a good lift of $\g$, then $\tilde \g$ certainly intersects  $\tilde S^+$.  Now referring to Figure~\ref{fig:cxfunddomain}, let $B_2,B_3$ be the circles 
with equal  diameters $2/\Im \tau_i$ tangent to $\RR$ at $0$,  and $\RR + \tau_i$ at $  \tau_i$, respectively. It follows from the usual ping-pong theorem methods,  that
there is a lift $ \tilde T$ of the axis of $T$  to $\dd \C^+$ which has one endpoint inside  $B_2$ and one inside $B_3$. This lift also clearly cuts $\tilde S^+$. By Lemma~\ref{lemma:reltwist1},  $i_{\s_i}(\g,T)$ is up to sign the number of images $S_i^n(\tilde \g)$  of $\tilde \g$ which cut $\tilde T$. Since $\tilde \g$ is a good lift , orienting as in Lemma~\ref{lemma:reltwist1}, we see that $i_{\s_i}(\g,T) =[\Re \tau_i/2] + O(1)$.  
We deduce that $$\Re \tau_i = -2p_i(\g) / q_i(\g) + O(1).$$
and the result follows.    \qed

\medskip

 \noindent{\sc Proof of Theorem~\ref{thm:cnvgtopml}}.
As usual let $\xi  \in \MLQ$  be  admissible and let $G_{\xi}(\th)$ be the unique group for which $\b(G) = \th \xi$.
 Let $h (\th)$ denote the hyperbolic structure of $\dd \C^+/G_{\xi}(\th)$.  
Since $l^+_{\s_i} \to 0, i=1,2$, the limit of the structures $h (\th)$
in $\PML$ is in the linear span of $\d_{\s_1},\d_{\s_2}$.  We want to prove that the limit is the barycentre $\d_{\s_1}+\d_{\s_2}$.

  Let $\d, \d' \in \S$.    Since $\s_1,\s_2$  are a maximal set of simple curves on $\Sigma$,  the thin part of $h(\th)$ is contained in  collars $A_i$ around $\s_i$ of approximate width 
  $ \log (1/l^+_{\s_i})$ and  the lengths of $\d, \d'$ outside the collars $A_i$ are bounded (with a bound depending only on the combinatorics of $\d,\d'$ and hence the canonical coordinates $\i(\d), \i(\d')$). By Proposition~\ref{prop:argpi/2} the  twisting around $A_i$ is bounded. We deduce that for any curve transverse to $\s_i$ we have 
\begin{equation}
\label{eqn:collar}
l^+_{\d} = \sum_{i=1,2} q_i(\d) \log (1/l^+_{\s_i}) + O(1),
\end{equation}
 see for example~\cite{minskyproduct} Lemma 7.2. By Theorem~\ref{thm:intromain1} we have $l^+_{\s_1} / l^+_{\s_2}\to q_2(\xi)/q_1(\xi)$, and since $\xi$ is admissible,  $q_1(\xi), q_2(\xi)>0$. Thus $ \log l^+_{\s_1} / \log l^+_{\s_2}\to 1$. Hence 
$$l^+_{\d}/ l^+_{\d'} \to \sum_{i=1,2} q_i(\d)/ \sum_{i=1,2} q_i(\d') = i(\d, \s_1+ \s_2)/i(\d', \s_1+ \s_2).$$
The result follows from the definition of convergence to a point in $\PML$.
\qed 

\begin{remark}
{\rm The above length estimate above coincides with that coming from the top terms Theorem~\ref{thm:topterms}. Namely  from that formula we have 
\begin{equation}
\label{eqn:collar1}
\log \tr {\d} = \sum_{i=1,2} q_i(\d) \log (\tau_i) + O(1).
\end{equation} Since $l_{\s_i}^+$ is small, any transverse curve has definite length. Hence by for example Proposition 5.1 of~\cite{smallbend}, $l^+_{\d}$ is close to the hyperbolic length of the geodesic representative of $\d$ in $\HH^3/G$ and thus to $\log \tr \d$.
Since by Proposition~\ref{thm:lengthzero2i}, $  l^+_{\s_i}  \Im \tau_i = O(1)$,
the formula \eqref{eqn:collar1} is  compatible with~\eqref{eqn:collar}. }
\end{remark}

\section{Asymptotic directions}
\label{sec:asympdirns}
In this section we  prove our main  results,
Theorems~\ref{thm:intromain} and~\ref{thm:intromaincompute}.

\medskip Throughout this section, to simplify notation, $X = O(\th)$ will mean $X \leq c\th$ where the constant $c$ depends on the lamination  $\xi $ and a small number of related curves chosen during the proofs.
With a bit more effort, the dependence could be controlled more carefully, but this is not needed for our results here. We write $\tr \g$ to mean $\tr \rho(W)$ where $W$ is a word representing $\g \in \pi_1(\S)$.
\medskip

Suppose that  $\gamma $ is a bending line of $\dd \C^+/G$ for a group $G(\tau_1,\tau_2) \in\P_{\xi}$. The top terms Theorem~\ref{thm:topterms}, together with the condition $\tr \g \in \RR$ of Lemma~\ref{lemma:realtrace}, gives
 asymptotic conditions  for   $(\tau_1,\tau_2) \in  \P_{\xi}$,  in terms of the canonical coordinates   ${\bf i}(\g)$ of $\g$. 
 For $\tau_1, \tau_2 \in \CC_+^2$ set  $\tau_i = x_i+i y_i, \rho  = \sqrt{y_1^2 + y_2^2} $, and 
 $\eta_i= y_i/ \rho $.  Define $$E_{\gamma} (\tau_1,\tau_2) = 
(q_1x_1 +2p_1
)\eta_2+ (q_2x_2 +2p_2)\eta_1,$$
where as usual ${\bf i}(\g) = (q_1, p_1, q_2, p_2)$ and $ y_i> 0, i=1,2$. 

\begin{proposition}
\label{prop:keyapprox} Suppose that $\xi \in \MLQ$ is an admissible lamination, that  $G(\tau_1,\tau_2)  \in \P_{\xi}$  has bending measure $\beta(G) = \th \xi$, and that $\g $ is a bending line of $\xi $. Then $E_{\gamma} (\tau_1,\tau_2) = O(\th)$ as $\th \to 0$.
\end{proposition}
\begin{proof} Suppose first that $q_i = q_i (\g)>0, i=1,2$ and set $a_i=-2 {p_i(\g)}/{q_i(\g)}$. 
By  Theorem~\ref{thm:topterms} we have
\begin{equation}
\label{eqn:topterms}
\tr \g=   \pm 2^{|q_2-q_1|}\bigl(\tau_1-a_1\bigr)^{q_1}
\bigl(\tau_2-a_2\bigr)^{q_2}
+R\bigl(q_1+q_2-2\bigr) 
\end{equation}
where  $R(q_1+q_2-2)$ is a polynomial of degree at most $q_i$ in
$\tau_i$  but with total degree in $\tau_1$ and $\tau_2$ at most
$q_1+q_2-2$.

By Proposition~\ref{prop:argpi/2}, $x_i - a_i = O(1)$ and by Theorem~\ref{thm:intromain1}:  
\begin{equation}
\label{eqn:ys}
\bigl |q_2(\xi)/q_1(\xi) -  \eta_1/ \eta_2\bigr |=   O(\th). \end{equation} 
(Notice that the terms in  \eqref{eqn:ys} involve $q_i(\xi)$ as opposed to $q_i = q_i(\g)$ in 
 \eqref{eqn:topterms}.)
Hence arranging the terms of \eqref{eqn:topterms} in order of decreasing powers of $\rho$,
 and using Equation~\eqref{eq:qcong}
in Section~\ref{sec:simplecurves}, we get
\begin{equation*}
\label{eqn:topterms1}
\begin{split}
 \pm  \tr & \g  2^{-|q_2-q_1|} =   \\
&  \rho^{q_1+q_2} {\eta_1}^{q_1} {\eta_2}^{q_2} +  i \rho^{q_1+q_2-1}{\eta_1}^{q_1-1} {\eta_2}^{q_2-1} (q_2\eta_1(x_2-a_2)  \\ & +q_1\eta_2(x_1-a_1) )    + O(\rho^{q_1+q_2-2}). 
 \end{split}
 \end{equation*}

By Lemma~\ref{lemma:realtrace}, $\tr \g \in \RR$. We deduce
${\eta_1}^{q_1-1}
 {\eta_2}^{q_2-1} E_{\g}(\tau_1,\tau_2) + O(1/\rho) = 0$
from which using~\eqref{eqn:ys},
$$q_2\eta_1(x_2-a_2) +q_1\eta_2(x_1-a_1) ) =   O(1/\rho).$$
Since $1/\rho = O(\th)$ by Corollary~\ref{cor:noalglim},  this proves
 the result.
 
We still have to deal with the case that, say,   $q_2(\g) = 0$. Then  $\tr \g $ is a polynomial in $\tau_1$ only, 
of the form
\begin{equation}
\label{eqn:topterms2}
\tr \rho(\g)=   \pm 2^{q_1}\bigl(\tau_1-a_1\bigr)^{q_1}
+R\bigl(q_1-2\bigr). 
\end{equation} 
The result then follows easily by similar reasoning to the above.  
 \end{proof}

\subsection{Solving the asymptotic equations}
Suppose we want  to locate the pleating ray $\P_{\g}$ of $\g \in S$.
If $G \in \P_{\g}$, then $ \dd \C^+/G  \setminus \g$ is flat, so that  not only $\gamma$, but also any curve $\d \in wh(\g)$,  is a bending line of $\xi$, where 
 $wh(\g)$ (the wheel of $\g$) denotes the set of  all curves $\d \in \S$ disjoint from $\g$.
(By convention, $ \g \notin wh(\g)$.) Thus $\tau_1,\tau_2$ are constrained by the equations 
$$\Im \tr \g =  \Im \tr \d =0 $$  and hence, by the above proposition,
\begin{equation*}
E_{\gamma} (\tau_1,\tau_2) +   O({\th}) = 0,  \ \ {\rm and} \ \ 
E_{\d} (\tau_1,\tau_2) +  O({\th})= 0 
\end{equation*} 
for all $\d \in wh(\g)$.  Our proof of   Theorem~\ref{thm:intromain}  amounts to solving these equations for $\tau_1,\tau_2$.

\medskip
In  order to do this, note that for any curve $\omega \in \S$:
$$E_{\omega}(\t_1,\t_2)=  \i(\omega)\cdot {\bf u} $$ where 
$\i(\omega)=  (q_1(\omega), p_1(\omega), q_2(\omega), p_2(\omega))$ and 
\begin{equation}
\label{eqn:xi}
{\bf u}  = (x_1 \eta_2, 2 \eta_2, x_2  \eta_1, 2\eta_1) 
 \end{equation} 
 with 
$ x_i = \Re \t_i, \eta_i = \Im \t_i /\rho$  as above.  
Effectively what we will do is use linear algebra to solve the equations $ \i(\d)\cdot {\bf u}  = 0$ for all $\d \in wh(\g)$. This is done with the aid of Thurston's symplectic form  $\Omega_{\Th} $ introduced in Section~\ref{sec:symplform}.
This induces a map $ \xi  \to \xi^*$ of $\RR^4$ such that 
$$\Omega_{\Th} (\i(\g),\i(\delta) ) =\ \i(\g)\cdot\ \i(\delta)^* $$
where $\cdot$ is the usual inner product on $\RR^4$: 
if   ${\bf i}(\g) = (q_1, p_1, q_2, p_2)$, then
 $\i(\g)^*  = (-p_1, q_1, -p_2, q_2)$. By Proposition~\ref{prop:symplform}, $\i(\g)^*$ is orthogonal not only to $\i(\g)$, but also to all curves in $wh(\g)$. We need:

 \begin{lemma} \hspace{-.5cm} 
\label{lem:3indeppeople} 
\begin{itemize}
\item[(i)] Suppose that $\gamma, \delta  \in \S$ and that $\delta \in wh(\g)$. Then  $\gamma$ and $ \delta$ are supported on a common canonical train track and $\i(\g), \i(\d)$ are independent vectors in $ML$.
\item[(ii)]  Given $\gamma \in \S$, we can  find 
$ \delta, \delta' \in wh(\gamma) $ such that $\i(\gamma),\i(\delta),\i(\delta')$ are supported on a common  canonical train track and span a subspace of dimension $3$ in $\ML$. 
\end{itemize}
\end{lemma}
\begin{proof} 
 \noindent $(i)$ That  $\g,\d$ are supported on a common canonical train track  follows immediately since they are disjoint. If   $\i(\gamma),\i(\delta)$ are dependent, since they lie on the same track all their coefficients are  integers which pairwise have the same sign. Thus we must have $n\i(\gamma) = m \i(\delta)$ for some $n,m \in \NN$.
Since both are connected simple curves,  $n= m =1$. 

\noindent$(ii)$  The surface $\Sigma_{\g} := \Sigma \setminus \g$ is either a one holed torus and a sphere with two punctures and a hole; or a sphere with two holes and two punctures. In either case, $\PML(\Sigma_{\g})$ is a topological circle in which the set of rational laminations supported on a simple curve is dense.  Let $\pi$ denote the map which associates to a curve $\omega \in \S$ the measured lamination $\d_{\omega} \in \ML(\Sigma)$, and likewise define the quotient map     $\bar \pi(\omega) =  [\d_{\omega}] \in \PML$. Then $\pi, \bar \pi$ are injective and 
  $\bar  \pi(wh\g)$ is a dense subset of an embedded circle $K$ in $ \PML(\Sigma)$. 

If  (ii) is false, then in particular $  \pi(wh(\g)) \subset \ML$  is contained in a $2$-plane whose image  in  $\PML$ is an affine  line $L $. (Since the canonical coordinates are global coordinates for $\ML$,  there is no need to consider the subdivision of $\ML$ into cells.) 
Since  $\PML(\Sigma_{\g})$ injects into $  L$, it  is open and closed in $K$.  Thus $K = L$, which is impossible. Thus $  \pi(wh(\g))$ spans at least a $3$-dimensional subspace in $\ML$.

 Suppose that $\pi(\g)$ is in the linear span of  $\pi(\d),\pi(\d') \in \pi(wh(\g)) \subset \ML$. Then by the above we can find $\d'' \in wh(\g) $ with $\pi(\d'')$ not contained in the linear span of  $\pi(\g),\pi(\d)$ and $\pi(\d')$ which proves the result.
 \end{proof}
  
 \medskip

 \begin{theorem} [Theorem~\ref{thm:intromain}]
   \label{thm:intromaindetail}
   Suppose that $\xi   =  \sum_{1,2}a_i {\g_i}$ is admissible and not exceptional.
Let $ \i(\xi) = (q_1,p_2,q_2,p_2)$ and 
set $ \tan \psi (\xi) = q_1/q_2$. Let 
 $L_{\xi}:   [0,\infty) \to  \CC^2$ be the line
$t \mapsto (w_1(t), w_2(t)) $ where $$w_1(t) = -2 p_1/q_1 + t \cos \psi, w_2(t) = -2 p_2/q_2 + t \sin \psi .$$
Let $(\tau_1(\th), \tau_2(\th)) \in \CC^2$  be the point corresponding to the group 
$G_{\xi}(\th)$ with $\beta(G) = \th \xi$, so that the pleating ray $\cal P_{\xi} $
is the image of the map $p_{\xi}: \th \to (\tau_1(\th), \tau_2(\th))$ for a suitable range of $\th>0$.  
  Then $\cal P_{\xi}$ approaches $L_{\xi}$ as $\th \to 0$ in the sense that
  if   $t(\th)  =  4Q/\th q_1q_2 $ with $Q = \sqrt{ (q_1^2 + q_2^2)}$ 
then 
$$  | \Re \tau_i(\th)  -   \Re w_i(t(\th))  | = O(\th) \  {\rm and}\ 
   | \Im \tau_i(\th)  -    \Im w_i(t(\th)) | = O(1),   i=1,2    .$$ 
 \end{theorem}
\begin{proof} As above, write $\tau_i(\th) = \tau_i = x_i + i y_i, \rho^2 = y_1^2 + y_2^2$ and $ \eta_i = y_i /\rho$, where the dependence on $\th$ is understood. By Theorem~\ref{thm:intromain1}, we have  
$y_i - \dfrac{4}{\th q_i} = O(1)$. 

On the other hand, with $t = t(\th)$ as in the statement of the theorem, we find
$ \Im w_1(t) = t q_2/Q =   4 /\th q_1$ and similarly $ \Im w_2(t) = 4 /\th q_2$.
Thus  $ | \Im \tau_i(\th)  -    \Im w_i(t(\th)) | = O(1),   i=1,2  $ as $\th \to 0$.

\medskip 
To deal with the coordinates  $x_i = \Re \tau_i(\th)$ is more subtle. Consider first the 
 special case $\xi = \g \in \S$.
By Lemma~\ref{lem:3indeppeople} we can choose 
$ \delta, \delta' \in wh(\gamma) $ such that $\i(\gamma),\i(\delta),\i(\delta')$ span a subspace of dimension $3$ in $\ML$. 
If $(\tau_1,\tau_2) \in \P_{\g}$, the conditions
$$\Im \tr \g = \Im \tr \d = \Im \tr \d' =0 $$ must be satisfied.
By Proposition~\ref{prop:keyapprox}, we can write this as 
$$E_{\zeta} (\tau_1,\tau_2)  + O( {\th}) =0 $$ for $\zeta= \g,\d,\d'$.
With $ \tau_i = x_i + i \rho \eta_i$ as above,   we can as in~(\ref{eqn:xi}) regard  these as equations in $\RR^4$
 for $
{\bf u} ={\bf u}(\th)=  (x_1 \eta_2, -2 \eta_2, x_2  \eta_1, -2\eta_1)$:
\begin{equation}
\label{eqn:almostzeta}
\i(\zeta) \cdot {\bf u}= O( {\th})
\end{equation}   for $\zeta= \g,\d,\d'$.
 
Now we use Thurston's symplectic form. By Proposition~\ref{prop:symplform} we  have $\Omega_{\Th} (\i(\gamma), \i(\zeta)) = 0$ for all $\zeta \in wh(\gamma)$. Hence  
 $$ \i(\zeta)\cdot \i(\gamma)^*=   0$$  for $\zeta= \g,\d,\d'$.
Since $\i(\gamma), \i(\d), \i(\d')$ are independent, it follows that we can write
\begin{equation}
\label{eqn:almost}
{\bf u} (\th) = \lambda (\th) \i(\gamma)^* +  \mu (\th) {\bf v}(\th)
\end{equation}  
where  ${\bf v} = {\bf v}(\th) $ is in the linear span of $\i(\gamma),\i(\d), \i(\d')$ and 
$||{\bf v} || = 1$. We find  using \eqref{eqn:almostzeta} that 
${\bf u}  \cdot{\bf v}=  O(  \th)$ (where the constants depend on 
$\i(\gamma), \i(\d), \i(\d')$). Then ${\bf v} \cdot  \i(\gamma^*) =0$ gives  $\mu (\th) =  O(  \th)$, with the same proviso on the constants. 
 Equating the two sides  of~\eqref{eqn:almost}
gives 
\begin{equation}
\label{eqn:almost*}
\begin{split}
&x_1\eta_2 = -\lambda  p_1(\g) +O(  \th), \ \  2\eta_2 = \lambda   q_1(\g) +O(  \th),\\  &x_2\eta_1 = -\lambda   p_2 (\g)+O(  \th),\ \  2\eta_1  = \lambda   q_2(\g) +O( \th).
\end{split}
\end{equation}

Since $
||{\bf u}||^2 = (x_1^2+4)   \eta^2_2 + (x_2^2+4)\eta^2_1 \geq 4$ we find
$$|\lambda(\th) | = \frac{||{\bf u} (\th)- O( \th) ||}{||\i(\gamma)^*||} \geq c $$
for some constant $c>0$. It follows easily that  $ |x_i  +2p_i/q_i| = O(\th)$, proving Theorem~\ref{thm:intromain}  in the special case $\xi = \g$.  

\medskip

Now we turn to the case of a general admissible lamination $\xi = a_{\g}  \g + a_{\d} \d \in \MLQ$.
In this case, if $(\tau_1,\tau_2) \in \P_{\xi}$ then $\g$ and $\d$ are both bending lines
of $G(\tau_1,\tau_2)$. It follows as above that 
$$\i(\gamma) \cdot {\bf u}=O( {\th})  \ \ {\rm and} \ \ 
 \i(\d) \cdot {\bf u} = O( {\th}) .$$
 By Lemma~\ref{lem:3indeppeople}, $ \i(\gamma)$ and $\i(\d)$ are independent and 
by Proposition~\ref{prop:symplform},  $ \i(\gamma)^*$ and $\i(\d)^*$ are orthogonal to both.
 Thus 
\begin{equation}
\label{eqn:almost1}
{\bf u} = \lambda  \i(\gamma)^* +    \mu  \i(\d)^*+ \nu {\bf w} 
\end{equation} 
where ${\bf w} $ is in the span of $ \i(\gamma),\i(\d)  $,  and $||{\bf w} || = 1$. 
Thus 
${\bf u}  \cdot{\bf w}=  O(  \th)$  and since ${\bf w} \cdot  \i(\gamma)^* ={\bf w} \cdot  \i(\d)^* =0$ we find   $\nu = \nu ||{\bf w} || ={\bf u}  \cdot{\bf w} = O(  \th)$.

Now from Theorem~\ref{thm:intromain1}, 
\begin{equation}
\label{eqn:almost2}
\Bigl |\frac{y_2}{y_1} - \frac{a_{\g} q_1(\g) + a_{\d} q_1(\d) }{a_{\g} q_2(\g) + a_{\d}  q_2(\d)}\Bigr | = O(\th).
\end{equation}  On the other hand
setting $z_i = \lambda q_i(\g) + \mu q_i(\d)$, we find from \eqref{eqn:almost1}
that $|z_1|^2 + |z_2|^2 \geq c >0$ for some constant $c$ depending only on $\g,\d$.
It follows from~\eqref{eqn:almost1} and~\eqref{eqn:almost2} that 
$|z_2/z_1  - y_2/y_1| = O(\th)$.
Since by hypothesis we are not in the exceptional case $ q_1(\g)/q_2 (\g)= q_1(\d)/q_2(\d)$, we  deduce that $\bigl |   \lambda/\mu  - a_{\g}/ a_{\d} \bigr | = O(\th)$.  Substituting back in  \eqref{eqn:almost1}  we see ${\bf u} = \kappa i(\xi)^* + O(\th)$ for some $\kappa >0$, and a little bit of algebra  completes the proof.   \end{proof}

\medskip

The following lemma proves the second part of the statement of  Theorem~\ref{thm:intromaincompute}.
\begin{lemma}
\label {lem:transverse} 
Suppose that $\g_1 \in \S$ is admissible. Then there exists
$\g_2 \in wh(\g_1)$ such that the pair  $\g_1,\g_2$ is not exceptional.
\end{lemma} 
\begin{proof}  Write  ${\bf X} = (X_1,X_2,X_3,X_4) \in \RR^4$. 
Then $\d \in wh(\g_1)$ implies that 
$\i(\d)$ is in the codimension one hyperplane $H $ defined by $\i(\g_1)^* \cdot {\bf X} = 0$.
If the pair $\g_1,\d$ is exceptional then $ \i(\d)$ is also in the 
codimension one hyperplane $K$ defined by $q_1(\g) X_1 -  q_2(\g) X_3 = 0$.
Note that $\i(\g)$ is also in $H  \cap K$.

Now since $q_1(\g),  q_2(\g)>0$, the normal vectors $$(-p_1(\g), q_1(\g),-p_2(\g),  q_2(\g)) \ {\rm and} \ (q_1(\g), 0,   q_2(\g),0)$$ to $H$ and $K$ respectively  are not collinear.
Thus $H  \cap K$ is two dimensional.   
However by Lemma~\ref{lem:3indeppeople}, we can find $\d,\d' \in wh(\g_1)$ such that 
$\i(\d),\i(\d'), \i(\g_1)$  are independent, so at least one of the pairs 
$  \i(\g_1), \i(\d)$ and $ \i(\g_1), \i(\d')$ must be non-exceptional as claimed.
 \end{proof}

 \noindent{{\sc Proof of Theorem~\ref{thm:intromaincompute}.}} 
 We need to show that if  $\xi   =  \sum_{1,2}a_i \delta_{\g_i}$ is admissible and if the pair  $\g_1, \g_2$ is not exceptional, then the equations
 $$ \tr \g_1, \  \tr \g_2  \in \RR$$
have a unique solution as $\tau_i \to \infty$ in the direction specified by Theorem~\ref{thm:intromain}.
First consider the equations
\begin{equation}
\label{eqn:ideal}
\tau_1^{q_1} \tau_2^{q_2} , \  \tau_1^{q_1'} \tau_2^{q_2'} \in \RR,  
\end{equation}
as  $\tau_1, \tau_2 \to \infty$, 
 where $q_i = q_i(\g_1)$ and $q'_i = q_i(\g_2)$. 
Setting $z_i = 1/\tau_i$  we define $G:\CC^2 \to \CC^2$ by 
$G(z_1,z_2) = (z_1^{q_1} z_2^{q_2},z_1^{q_1'} z_2^{q_2'}) $, so that~\eqref{eqn:ideal} is equivalent to the equation $G(z_1,z_2) \in \RR^2$ as $z_1,z_2 \to 0$.
  Writing $z_i = \epsilon_i e^{i\th_i}, i=1,2$ we obtain
\begin{equation}
\label{eqn:ideal1}
  e^{i(q_1\th_1+ q_2 \th_2)}, \ e^{i(q_1'\th_1+ q_2' \th_2)}  \in \RR,
\end{equation}
so that writing   $ A = \begin{pmatrix}  q_1 & q_2 \\ q_1' & q_2'   \end{pmatrix} $, we must have $ A (\th_1,\th_2)^T = \pi (n,m)^T$ for 
  $n,m \in \ZZ$.
Since $A$ is invertible by hypothesis, these equations have a discrete set of solutions for $(\th_1,\th_2)$. Since $q_1+q_2, q_1' + q_2' \in 2 \ZZ$, one of these solutions is
$\th_i = \pi/2, i=1,2$. 
It follows that $G$ is a branched covering 
 $\CC^2 \to \CC^2$ in a neighbourhood of $0$, and $G^{-1}(\RR^2)$ is a discrete set of $2$-planes meeting only at $0$.

Now consider our actual equations $\Im \tr \g = \Im \tr \d = 0$.
As above set $z_i = 1/\tau_i$ and
define $H:\CC^2 \to \CC^2$ by 
$H(z_1,z_2) =(1/\tr \g, 1/\tr \d) $. By Theorem~\ref{thm:topterms} we have 
\begin{equation}
\label{eqn:actual}
\begin{split} \tr \g &=(\tau_1 + 2p_1/q_1)^{q_1}(\tau_2  +2p_2/q_2)^{q_2}(1 + R_1 ),\\  \tr \d &= (\tau_1 + 2p'_1/q'_1)^{q'_1}(\tau_2 + 2p'_2/q'_2)^{q'_2}(1 + R_2 ) .
\end{split}
\end{equation}
where $p_i= p_i(\g), q_i = q_i(\g) $ and $p'_i= p_i(\d), q'_i = q_i(\d)$
and  $  R_1,   R_2$ denote polynomials of total order at most $q_1+q_2 -2$ in $\tau_1, \tau_2$.
Hence we can expand $H(z_1,z_2) $ as a Taylor expansion about $0$ to obtain
\begin{equation}
\begin{split}
\label{eqn:actual1}
H(z_1,z_2) =  &\\ 
(z_1^{q_1}z_2^{q_2}&(1 + p_1 z_1 + p_2 z_2 + \hat R_1),  z_1^{q'_1}z_2^{q'_2}(1 + p'_1 z_1 + p'_2 z_2 + \hat R_2))
\end{split}
\end{equation}
where $ \hat R_1, \hat R_2$ denote terms of total order at least $2$ in $z_1,z_2$ and  $p_i = p_i(\g), p'_i = p_i(\d)$.

There is clearly a neighbourhood $U$ of $0$ such that   $H$ is locally injective on $U \setminus \{0\}$, and moreover in which 
 the homotopy  $ G + tH, t \in [0,1]$
 between $G$ and $H$  is regular at every point. It follows that  $H$ is also a branched covering of $\CC^2$ near zero, of the same order as $G$, and that there is a natural bijective correspondence between the branches of $G^{-1}(\RR^2)$ and  $H^{-1}(\RR^2)$. To complete the proof, we need to show that
for $(t_1,t_2) \in \RR^2$ sufficiently near $0$, the point  $H^{-1}(t_1,t_2)$ is arbitrarily close to the point $G^{-1}(t_1,t_2)$
on the corresponding sheet of  $G^{-1}(\RR^2)$.

In a neighbourhood of $0$, we view \eqref{eqn:actual1}
 as a perturbation of 
\eqref{eqn:ideal} and use the ideas of Appendix B in~\cite{Mil}. If $g: \CC^n \to \CC^n$
is an analytic function with an isolated zero at $Z_0  \in \CC^n$,  
define  the \emph{multiplicity} of $g $ to be the degree of the mapping 
$ g/||g|| : S_{\d} \to S_1$, where $S_{\d} $ is the sphere radius $\d$ centre $Z_0$ and $S_1$ is the unit sphere.   The same proof as Lemma B.1 of~\cite{Mil} proves `Rouch{\'e}'s principle'  that  if $ r : \CC^n \to \CC^n$ with $r(0) = 0$ and if  $||r|| < ||g||$ on $S_{\d} $, then  the degrees of 
 $(g + r)/||g + r||$ and $ g/||g||$  on $S_{\d} $ are equal.

Now take $Z_0 \in U$ to be an isolated solution  of the equation $G(z_1,z_2)  =(t_1,t_2) \in \RR^2 $. Choose $\d>0$ so that so that $S_{\d}(Z_0) \subset U$ and so that $G(z_1,z_2)  =(t_1,t_2)$ has  no other solutions in $S_{\d}(Z_0)$.  We can also choose $U$ small enough so that  $||(H  -G)|| < || G||$ on $U$.
It follows from   the above Rouch{\'e}'s principle that $ H$ and $G$ have the same degree
on $S_{\d} $.  Then Lemma B.2 of ~\cite{Mil} shows that $ H$ has exactly one zero inside  $S_{\d} $ as required.

In particular, there is a  unique branch of $H^{-1}(\RR^2)$  close to the branch
$z_i = \epsilon_i i, i=1,2$ of $G(z_1,z_2) \in \RR^2$.  If  $\epsilon_1/\epsilon_2$ is bounded away from $0$ and $\infty$ and we set $\ep  = \sqrt{\epsilon^2_1+ \epsilon^2_2}$, we can clearly write points on this branch  in the form
$z_i = \epsilon_i ie^{i\alpha_i}, i=1,2$ where  $\a_i = O(\ep)$ as $\ep  \to 0$.
The arguments of Theorem~\ref{thm:intromain}  are then sufficient to show the solution to the equations  in that theorem is unique.
\qed

  \subsection{The exceptional case.}
 \label{sec:exceptional}

 Recall that a pair of curves $\gamma_1,\g_2$ is said to be exceptional  if 
 $q_1(\g_1)q_2(\g_2) = q_1(\g_2)q_2(\g_1)$. As an example, the  
 coordinates 
  $\i(\g_1) = (2,0,2,1)$ and $\i(\g_2) = (2,-1,2,2)$, or more generally 
 $\i(\g_1) = (a+b,0,a+b,a)$ and $\i(\g_2) = (a+b,-1,a+b,a+1)$, can be easily be checked to represent  exceptional pairs of disjoint connected curves.

   \begin{theorem} 
   \label{thm:exceptional}
Suppose that $\xi   =  \sum_{1,2}a_i  {\g_i}$ is an admissible lamination such that  the pair $\gamma_1,\g_2$ is exceptional. For $ s \in [0,1]$ let $\xi(s) =  \sum_{1,2}sa_1  {\g_1} + (1-s)a_2  {\g_2}$. Let  $ \i(\xi(s)) = (q_1(s),p_2(s),q_2(s),p_2(s))$ and 
set $ \tan \psi (s) = q_1(s)/q_2(s)$. Let 
 $L_{\g_1,\g_2}:  [0,1] \times [0,\infty) \to  \CC^2$ be the map
 $(s,t)\mapsto (w_1(s,t), w_2(s,t)) $ where $$w_1(s,t) = 2 p_1(s)/q_1(s) + t \cos \psi(s), w_2(s,t) = 2 p_2(s)/q_2(s) + t \sin \psi (s).$$

Let $(\tau_1(s,\th), \tau_2(s,\th)) \in \CC^2$  be the point corresponding to the group 
$G_{\xi(s)}(\th)$ with $\beta(G) = \th \xi(s)$, so that the  (closure of the) pleating plane $ \cal P_{\g_1,\g_2}$
is the image of the map 
$p_{\g_1,\g_2}: (s,\th) \to (\tau_1(\th), \tau_2(\th))$ for $s \in [0,1]$ and a suitable range of $\th>0$.  
  Then $\cal P_{\g_1,\g_2}$ approaches $L_{\g_1,\g_2}$ as $\th \to 0$ in the sense that
  if   $t(s,\th)  =  4Q/\th q_1(s)q_2(s) $ with $Q = \sqrt{ (q_1^2(s) + q_2^2(s))}$, then for all sufficiently small $\th$  
 there exists  a continuous function $f_{\th}:  [0,1] \to  [0,1]$ such that $f_{\th}(0) = 0, f_{\th}(1) = 1$  and 
\begin{equation}
\label{eqn:planesclose}
\begin{split} 
& | \Re \tau_i(s, \th)  -   \Re w_i(f_{\th}(s), t(s,\th))  | = O(\th) \  {\rm and}   \\ 
 &  | \Im \tau_i(s,\th)  -    \Im w_i(f_{\th}(s),t(s,\th)) | = O(1) 
 \end{split} 
 \end{equation}
   for $ i=1,2 $. Moreover every point on $L_{\g_1,\g_2}$ is close to a point on $\cal P_{\g_1,\g_2}$, in the sense that  for all sufficiently small $\th$, for each $s$ there exists $s' \in [0,1]$ such that 
 \begin{equation}
\label{eqn:planesclose1}
\begin{split} 
& | \Re \tau_i(s', \th)  -   \Re w_i(s, t(s',\th))  | = O(\th) \  {\rm and}   \\ 
 &  | \Im \tau_i(s',\th)  -    \Im w_i(s,t(s',\th)) | = O(1) 
 \end{split} 
 \end{equation}  
 \end{theorem}

\begin{remark}
 {\rm The difference between this statement and that of Theorem~\ref{thm:intromaindetail}
 is that in that theorem,  $\P_{\xi(s)}$  is close to a point on $L_{\xi(s)}$, while here we can only assert closeness of points
 on the pleating ray $\P_{\xi(s)}$ to  a point on \emph{some} line $L_{\xi(f_{\th}(s))}$ where possibly  $f_{\th}(s) \neq  s$.}
\end{remark}
 \begin{proof}  
 We proceed as in the proof of Theorem~\ref{thm:intromaindetail}
above, up to \eqref{eqn:almost1} which says that 
 ${\bf u} ={\bf u}(\th) $ is close to a convex combination of  $\i(\g_1)^* $ and $  \i(\g_2)^*$. Define $f_{\th(s)}$ by letting  $ \xi(f_{\th(s)})$  be the orthogonal projection of ${\bf u}(s,\th)$ onto the plane spanned by $\i(\g)^*, \i(\d)^*$.
Then  \eqref{eqn:planesclose}  follows  as before, while continuity of the path  $ s \mapsto f_{\th}(s)$  is clear. 
Equation~\eqref{eqn:planesclose1} follows choosing $s'$ such that $s= f_{\th}(s')$.
  \end{proof}

\begin{remark}
 \label{rem:exceptional1} {\rm One might expect to be able to prove Theorem~\ref{thm:intromain}   in the exceptional case by a limiting argument with laminations $\xi_n \to \xi$. However the interaction of the double limits
as $n \to \infty$ and $\th \to 0$ is quite subtle and we have not been able to extract the required results by this method.}
\end{remark}

\begin{remark}
 \label{rem:exceptional} {\rm Suppose that  the pair $\gamma_1,\g_2$ is exceptional.
 Then we claim that it is not possible to deduce from the top terms Theorem~\ref{thm:topterms}
 alone that the equations $\Im \tr \g_i = 0, i=1,2$ have a unique branch at infinity satisfying the conditions $\arg \tau_i = \pi/2,  \Im \tau_1/\Im \tau_2 = q_2(\xi)/q_1(\xi)$. To study  this question, as above we replace the equations  by equations $\Im f_1(z_1,z_2) = \Im f_2(z_1,z_2)  = 0$  in a neighbourhood of $0$, where $f_i: \CC^2 \to \CC^2$ are analytic functions with lowest order terms 
 $z_1^{q_1(\g_1)}z_2^{q_2(\g_1)}$ and $z_1^{q_1(\g_2)}z_2^{q_2(\g_2)}$ respectively.

To show that many different behaviours are possible, consider the functions
$f_0 (z_1,z_2) = z_1 z_2, f_1 (z_1,z_2) = z_1 z_2(1 - z_1 + z_2 +  z_1 z_2) ,f_2 (z_1,z_2) = z_1 z_2  (1 - z_1- z_2 +  z_1 z_2)$ and $f_3 (z_1,z_2) = z_1 z_2(1 - z_1 + z_2 +  z_1^2)$.
We look for solutions to  each of the three pairs of equations 
$\Im f_0 =0, \Im f_i =0, i=1,2,3$ which satisfy $ z_1=  \epsilon ie^{i \alpha}, z_2= \epsilon ie^{-i \alpha}$ where $\epsilon \to 0$ and $ \alpha= O(\epsilon)$.

As is easily verified, the equations $\Im f_0 =0, \Im f_1 =0$ are satisfied for arbitrary choices of $\alpha$.
 The equations $\Im  f_0 =0, \Im f_2 =0$ have no solutions of the required form near $(0,0)$.
Finally  the equations $\Im  f_0 =0, \Im f_3 =0$ have a unique suitable solution for each $\epsilon>0$, namely $\alpha= 0$.
 } 
\end{remark}

  \section*{Appendix 1}
\label{sec:appendix1}
\noindent The following proofs are taken from~\cite{kps1}.

\smallskip

\noindent{\sc Proof of Proposition~\ref{thm:inM}}
  Let $S_3 = TS_2T^{-1}$, so that $S_3$ is parabolic with fixed point $T(0) = \tau_1$. Let $J_j=\langle S_j \rangle$ for $j=1,2,3$. We construct fundamental domains
$D_j$ for the $S_j$ as follows.

Referring to Figure~\ref{fig:cxfunddomain} in Section~\ref{sec:concrete}, 
  let $l$ consist of the vertical line below $-1+i/2$, the vertical line
above $\tau_1 -1-i/2$ and the straight line joining $-1+i/2$ to
$\tau_1 -1-i/2$. This last line segment has positive slope
because $\Im\tau_1 >1$, from which it follows that  $l$ and $S_1(l)$ do not intersect and hence that  the strip $D_1$ between $l$ and $S_1(l)$
is a  fundamental domain for $J_1$.
A fundamental domain for $J_2$ is
$$D_2 = \bigl\{z\in\Chat :|z+1/2| > 1/2, |z-1/2| > 1/2 \bigr\}, $$
and a fundamental domain for $J_3$ is
$$D_3 =  \bigl\{z\in\Chat :
|z-\tau_1+1/2| > 1/2, |z-\tau_1-1/2| > 1/2 \bigr\}.$$
The hypothesis $\Im\tau_1 >1$ implies that the union of the closure of any two of
the $D_j$ is the whole of $\Chat$. Moreover the boundaries of the $D_j$ only intersect at
parabolic fixed points. Therefore by a simple application of the first
Klein-Maskit combination theorem (see \cite{Maskitbook}, p.149 or \cite{beardon},
p.103) we see that $F_0=\langle S_1,S_2,S_3\rangle$ is discrete
with fundamental domain $D_0=D_1\cap D_2\cap D_3$.

Now let $J_4=\langle T\rangle$. We will construct a fundamental domain
$D_4$ for $J_4$. Let $B_2$ be the disk centred at $i/\Im\tau_2$ with radius
$1/\Im\tau_2$ and let $B_3$ be the disk centred at $\tau_1-i/\Im\tau_2$ with
radius $1/\Im\tau_2$. One  checks that $T$ takes $B_2$ to the
complement of $B_3$. (Note that $T(z) = \tau_1 + \frac{1}{\tau_2 + 1/z}$ and consider the action of  $ z \mapsto \frac{1}{\tau_2 + 1/z}$ on $B_2$.) Thus the domain $D_4$ exterior to both disks is a fundamental
domain for $J_4$. Since $\Im\tau_2>1$ and
$\Im\tau_1\Im\tau_2>4$, it is easy to see that $B_2$ and $B_3$ are contained in
the strip $D_1$, $B_2$ is contained in $D_3$ and $B_3$ is contained in
$D_2$. Moreover $S_j(B_j)=B_j$ for $j=2,3$. Thus $B_j$ is precisely invariant
with respect to $J_j$ for $j=2,3$;
that is, $W(B_j) = B_j$ for $W \in J_j$ and $W(B_j) \cap B_j = \emptyset$
for $W\in F_0-J_j$.
Therefore the hypotheses of the second combination theorem are satisfied and
so $G_0=\langle F_0,T\rangle$ is discrete with  domain
$D=D_1\cap D_2\cap D_3\cap D_4$.

 Now we verify that $G \in \M$. By construction, $D$ consists of three components. The first, in the
lower half plane $\HH^-$, is a component of a fundamental domain for the
Fuchsian subgroup $F_2 = \langle S_1,S_2\rangle$ and 
$\HH^-/F_2$ is a triply punctured sphere. Similarly the second, in  the
half plane $\HH^{\tau_1}$ above the horizontal through $\tau_1$,  is a
component of a fundamental domain for the Fuchsian subgroup
$F_3=\langle S_1,S_3\rangle$. Again $\HH^{\tau_1}/F_3$ is a triply punctured
sphere. Because the components of $D$ in $\HH^-$ and $\HH^{\tau_1}$ are disjoint,
 $F_2$ and $F_3$ are not conjugate in $G$. 

Let $D^*$ denote the  third component $D$. It is contained in the strip between the horizontal
lines through $0$ and $\tau_1$. It  has eight sides, one pair of sides contained in the boundaries of each of
$D_1$, $D_2$, $D_3$, $D_4$ and identified by $S_1$, $S_2$, $S_3$ and $T$
respectively. Performing these identifications we obtain a torus with
two punctures corresponding to $S_2S_1^{\ -1}$ and $S_3S_1^{\ -1}$.
Developing $D^*$ by $G$ we see that it corresponds to a simply connected
$G$-invariant component of the regular set of $G$, and we conclude that $G\in\M$.

Finally, since translates of $D$ by $S_1$ cover the strip $ 1/2 < \Im z < \Im \tau_1-1/2$, 
and since  $\HH^-, \HH^{\tau_1}$ are in $\Omega(G)$, the limit set $\Lambda$ is contained in  the two strips $ 0 \leq \Im z \leq 1/2$ and $ \Im \tau_1-1/2 \leq \Im z \leq \Im \tau_1$  as claimed.
\qed

\medskip

\noindent{\sc Proof of Proposition~\ref{thm:inM1}}
This is based on a similar result for the Maskit space of the once
punctured torus due to David Wright~\cite{Wright}.
Let $W \in G - \{\langle S_1,S_2\rangle\cup\langle S_1,S_3\rangle\}$ and
let $\HH^-$ be the lower half plane. Then $W(\HH^-)$ is a disk contained in the strip $\{0 < \Im z < \Im \tau_1\}$.
Let
$$
W=   \left(  \begin{array}{cc}
                 a  &  b \\
                 c  &  d
              \end{array}    \right)
$$
with $ad-bc=1$ and suppose that the circle  $C = W(\RR \cup \infty) $ has radius $r$ and centre $z_0$, so that $ \Im z_0 \geq r >0$. Using the fact that the points $T^{-1}(z_0), T^{-1}(\infty)$ are inverse points with respect to $\RR$ (see also~\cite{Indra} p.91), we find that  $r = i/(c\bar{d}-d\bar{c})$ and $z_0 = (a\bar{d}-b\bar{c})/(c\bar{d}-d\bar{c})$.
The inequality   $ \Im z_0  >0$ gives  $ \Im c\bar d >0$  and then $ \Im z_0 \geq r $ simplifies to  $\Re(b\bar{c}- a\bar{d}) \geq 1$.

Applying this to $T$ we see that $\Im \tau_1 \Im \tau_2 \geq 1$.
Applying it to
$$
[S_1,T^{-1}]= \left( \begin{array}{cc}
		1 - 2\tau_2 +4\tau_2^2 & 4\tau_2 \\
		2\tau_2^2 & 1+2\tau_2
		\end{array}  \right) $$
we see that $\Im \tau_2 \geq 1/2$, and similarly applying it to $[T,S_2^{\ -1}]$ we have $\Im\tau_1\geq1/2$.
\qed

\section*{Appendix 2}
\label{sec:appendix2}

To shed more light on the definition of canonical coordinates, we recall the familiar  situation for $\Sigma_{1,1}$, see~\cite{birmans}.
 Consider first a Euclidean torus $\Sigma_{1,0}$ viewed as the quotient of $ \CC$  by translations $S: z \to z+1$ and $T: z \to z+i$.  
The unit square $ 0 \leq \Re z, \Im z \leq 1$ is a fundamental domain $\Delta_0$ with opposite sides identified by  $S $ and $T $. Label each side  by the translation  which carries it  to a paired side, so that  the bottom side $\Im z = 0$ is labelled $s_T$ since $T$ carries it to the top side    $\Im z =1$;   similarly the left side $\Re z = 0$  is labelled  $s_S$ since $S$ carries it to the right side $\Re z = 1$. Any closed geodesic $\g$ on $\Sigma_{1,0}$ lifts to a line  in $\CC$.  The lifts of $\
\g$ cut  $\Delta_0$ in a number of pairwise disjoint parallel arcs running from one side of $\Delta_0$ to another: if the line has slope $p/q$ there are $q \geq 0$ strands which intersect $s_T$ and $s_{T^{{-1}}}$,  and $|p|$ strands which intersect $s_S$  and $s_{S^{{-1}}}$.  By collapsing all the strands which join a fixed pair of sides into one arc we obtain one of the four configurations
shown in Figure~\ref{fig:opttraintracks}, which we can view as four train tracks on  $\Sigma_{1,0}$.
Note that both the supporting train track and the weight on each branch  is completely determined by the \emph{signed} slope $p/q$. For example, if $p/q <-1$ and we choose $q \geq 0$,  there is one branch from 
$s_S$ to $s_{S^{-1}}$ of weight $|p|-q$ and there are branches joining 
$s_{S^{-1}}$ to $s_{T^{-1}}$ and $s_T$ to $s_{S}$ each of weight $q$, shown in the lower left hand quadrant of Figure~\ref{fig:opttraintracks}.

\begin{figure}[hbt] 
\centering 
\includegraphics[height=10cm, viewport =  350 150  400  700]{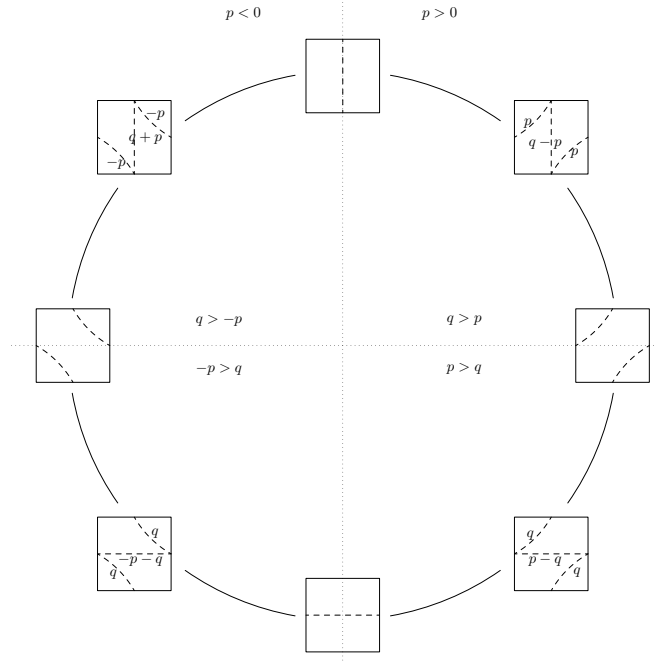} 
\caption{Canonical train tracks for the once punctured torus. In all charts, the total weight on the vertical sides is $|p|$ and on the horizontal sides is $q$.}
\label{fig:opttraintracks}
\end{figure} 

We obtain the canonical coordinates  for $  \ML (\Sigma_{1,1})$ by viewing $\Delta_0$ as a schematic representation of a fundamental domain for a hyperbolic once punctured torus $\Sigma_{1,1}$, with the puncture at the vertices. A simple closed curve on $\Sigma_{1,0}$ is also a  simple closed curve on $\Sigma_{1,1}$, and every simple closed curve on $\Sigma_{1,1}$ can be represented in this way~\cite{birmans}. (The proof is to  eliminate the possibility of non zero weights on all four corner arcs and then use the switch conditions as in Section~\ref{sec:simplecurves}.) Thus any such curve is supported on one of the four train tracks  in Figure~\ref{fig:opttraintracks} and 
the spaces of weights on these four tracks are a cell decomposition for    $  \ML (\Sigma_{1,1})$  in the usual way.

Let $\g_{p/q} \in \S (\Sigma_{1,1})$ be the curve associated to the line of slope $p/q$ in $\CC$. The canonical coordinates  $\i(\g_{p/q}) = (q,p)\in \ZZ_+ \times \ZZ$  of  $\g_{p/q} $ are  the \emph{signed} weights obtained  as above from  the original line in $\CC$.  We always take  $q = i(S,\g) \geq 0$. We take
 $p > 0$ if the line has  positive slope, that is, if  it contains an arc from $s_{S}$ to  $s_{T^{-1}}$ and take $p < 0$ if there is an arc from $s_{S^{-1}}$ to  $s_{T^{-1}}$.  (If $p=0$ the diagonal arcs have zero weight and we are on the boundary of two cells.)
It is easy to see from the above discussion that $\i(\g_{p/q})$ determines both a train track  track on $\Sigma_{1,1}$ and the weights on that track, thus giving  global coordinates for the homotopy classes of simple closed curves on $  \ML (\Sigma_{1,1})$.
Note that the coordinates $(|p|,q)$  of $\g_{p/q} $ are \emph{not} in general equal to  the weights on the  branches. In fact, up to the choice of a base point for the twist, $p/q$ is the Dehn-Thurston twist coordinate of $\g$. 
These coordinates extend naturally  by linearity and continuity to global coordinates for $  \ML (\Sigma_{1,1})$.

\section*{Appendix 3}
\label{sec:appendix3}
We give the (presumably well known) formula for the bending angle $\phi$ between two consecutive  segments of a geodesic $s$ on $\dd \C^+$ which  crosses a bending line $L$ of $\dd \C^+$ making an angle $\psi$ with $L$.  Let the bending angle between the two planes $\Pi_1,\Pi_2$ which meet along $L$ be $\th$. Let $s_i \subset \Pi_i $ be the two segments of $s$ which meet at  $P \in L$.   Measure $\phi$ so that  $\phi = \th$ when  $\psi = \pi/2$.
The  formula is 
$$ \sin \phi/2 = \sin \psi \sin \th/2.$$

This  can be proved by elementary Euclidean trigonometry. Let $N$ be the line perpendicular to $L$ which bisects the angle between $\Pi_1,\Pi_2$. The configuration of $L, \Pi_1,\Pi_2$ is invariant under rotation $\Omega$ by $\pi$ about $N$. Thus $\Omega$  interchanges $s_1$ and $s_2$ and  $N$ is contained in the plane of $s_1$ and $s_2$.


Thinking of $N$ as `vertical', let   $M$ be the  `horizontal' line through $P$ perpendicular to $L$ and $N$.   By definition of the bending angle and symmetry, $M$ makes an  angle $\th/2$ with $\Pi_2$. Then the segment $s_2$ 
makes an angle $\phi/2$ with the 
`horizontal'  plane $H$ spanned by $L$ and $M$.  
Choose $X$ to be a point on $L$ at distance $1$ from $P$. Assuming as we may that  $\psi \neq \pi/2$, let $Y \in \Pi_2$ be the point at which the perpendicular in $\Pi_2$ from $X$ meets $s_2$ and  let $Z$ be the foot of the perpendicular from $Y$ to  $H$. Then
$ \sin \phi/2 = |ZY|/|PY| =  |ZY| \cos \psi $ and $\sin \th/2=  |ZY|/|XY| =|ZY|/ \tan \psi$. The result follows.

 

\end{document}